\documentclass[a4paper,twoside,10pt,leqno]{article}
\usepackage[dvipsnames,x11names]{xcolor}
\usepackage{expl3}
\usepackage{amssymb,amsmath,amsthm,mathtools}
\usepackage[nottoc, notlof, notlot]{tocbibind}
\usepackage{stmaryrd}
\usepackage[english]{babel}
\usepackage{enumitem}
\usepackage[all]{xy}
\usepackage{tikz,tikz-cd}
\usetikzlibrary{decorations.pathreplacing}
\usetikzlibrary{decorations.markings}
\usetikzlibrary{arrows}
\usetikzlibrary{decorations.pathmorphing}
\usetikzlibrary{patterns}

\makeatother
\makeatletter
\def\blfootnote{\xdef\@thefnmark{}\@footnotetext}
\makeatother

\title{Motivic, logarithmic, and topological \\Milnor fibrations}
\author{
{\scshape Jean-Baptiste Campesato}\thanks{Univ Angers, CNRS, LAREMA, SFR MATHSTIC, F-49000 Angers, France \newline \hspace*{1.8em} E-mail: \texttt{jb.campesato@univ-angers.fr} \newline \ }\quad\quad
{\scshape Goulwen Fichou}\thanks{Univ Rennes, CNRS, IRMAR - UMR 6625, F-35000 Rennes, France \newline \hspace*{1.8em} E-mail: \texttt{goulwen.fichou@univ-rennes.fr} \newline \ }\quad\quad
{\scshape Adam Parusiński}\thanks{Université Côte d’Azur, CNRS, LJAD, France \newline \hspace*{1.8em} E-mail: \texttt{adam.parusinski@univ-cotedazur.fr} \newline \ }
}
\date{November 29, 2024}

\usepackage[inner=4cm,outer=3cm,top=4cm,bottom=4cm,includeheadfoot,headsep=.5cm]{geometry}
\usepackage{fancyhdr}
\renewcommand{\thepage}{\arabic{page}}
\usepackage{titlesec}
\titleformat{\subsection}[runin]
        {\normalfont\bfseries}
        {\thesubsection}
        {0.5em}
        {}
        [.]

\usepackage{hyperref}
\usepackage{color}
\definecolor{darkred}{rgb}{.5,0,0}
\definecolor{darkgreen}{rgb}{0,.5,0}
\definecolor{darkblue}{rgb}{0,0,.5}
\hypersetup{
    pdfencoding=unicode,
    colorlinks=true,
    linkcolor=darkred,
    citecolor=darkgreen,
    urlcolor=darkblue,
    bookmarks=false,
    pdftoolbar=true,
    pdfmenubar=true,
    pdffitwindow=true,
    pdfstartview={FitH},
    pdfcreator={},
    pdfproducer={},
    pdfnewwindow=true,
    pdftitle={Motivic, logarithmic, and topological Milnor fibrations},
    pdfauthor={Jean-Baptiste Campesato, Goulwen Fichou, Adam Parusiński},
}

\usepackage{stix}
\DeclareSymbolFont{calletters}{OMS}{cmsy}{m}{n}
\DeclareSymbolFontAlphabet{\mathcal}{calletters}
\usepackage{tgpagella}
\usepackage[T1]{fontenc}
\usepackage[utf8]{inputenc}

\newtheorem{thmx}{Theorem}

\newtheorem{thm}{Theorem}[section]
\newtheorem{lemma}[thm]{Lemma}
\newtheorem{prop}[thm]{Proposition}
\newtheorem{cor}[thm]{Corollary}
\theoremstyle{definition}
\newtheorem{defi}[thm]{Definition}

\newtheorem{ex}[thm]{Example}
\newtheorem{rmk}[thm]{Remark}

\newcommand{\LL}{\mathbb L}
\newcommand{\PP}{\mathbb P}
\newcommand{\II}{\mathbb 1}
\renewcommand{\L}{\mathcal L}
\newcommand{\X}{\mathfrak X}

\newcommand{\R}{\mathbb R}
\newcommand{\C}{\mathbb C}
\newcommand{\N}{\mathbb N}
\newcommand{\Z}{\mathbb Z}
\newcommand{\Rp}{\R_{>0}}
\newcommand{\Ro}{\R_{\ge0}}
\newcommand{\Ri}{(0,\infty]}

\newcommand{\Np}{\N_{>0}}
\newcommand{\OO}{\mathcal O}

\newcommand{\Sph}{\mathbb S}
\newcommand{\one}{\mathbf 1}
\newcommand{\oneplus}{\one_{\ge 0}}
\newcommand{\Do}{D^\circ}
\newcommand{\tildeDoQ}{\tilde D^\circ_{\tilde Q}}
\newcommand{\tildeDoK}{\tilde D^\circ_{\tilde K}}

\def\acts{\ensuremath{\rotatebox[origin=c]{-90}{$\circlearrowright$}}}
\newcommand{\fprod}[2]{\displaystyle\prod_{#2}{\raisebox{-.65em}{\text{\!\scriptsize$#1$}}}}

\usepackage{xspace}

\DeclareMathOperator{\ac}{ac}
\DeclareMathOperator{\ord}{ord}
\DeclareMathOperator{\Spec}{Spec}
\DeclareMathOperator{\Var}{Var}
\DeclareMathOperator{\pr}{pr}
\DeclareMathOperator{\id}{id}
\DeclareMathOperator{\ev}{ev}
\DeclareMathOperator{\Hom}{Hom}
\DeclareMathOperator{\can}{can}
\DeclareMathOperator{\iso}{iso}
\DeclareMathOperator{\tr}{triv}

\newcommand{\VarDCn}{\Var_{D\times\C^*}^{\C^*,n}}
\newcommand{\VarDC}{\Var_{D\times\C^*}^{\C^*}}
\newcommand{\K}{K_0}
\newcommand{\KDCn}{\K\left(\Var_{D\times\C^*}^{\C^*,n}\right)}
\newcommand{\KDC}{\K\left(\Var_{D\times\C^*}^{\C^*}\right)}
\newcommand{\M}{\mathcal M}

\newcommand{\MDC}{\M_{D\times\C^*}^{\C^*}}

\newcommand{\longhookrightarrow}{\lhook\joinrel\longrightarrow}

\newcommand{\sbar}[1]{\ensuremath{\overline{#1}^\Sph}}
\newcommand{\pbar}[1]{\ensuremath{\overline{#1}^{\mathbb P}}}

\newcommand{\mon}{\mathrm{mon}}
\newcommand{\mot}{\mathrm{mot}}
\renewcommand{\top}{\mathrm{top}}
\newcommand{\cpl}{\mathrm{cpl}}
\renewcommand{\log}{\mathrm{log}}
\newcommand{\clog}{\mathrm{clog}}
\newcommand{\mlog}{\mathrm{alog}}
\newcommand{\Mlog}{(M,D)^\log}
\newcommand{\tMlog}{(\tilde M,\tilde D)^\log}
\newcommand{\Mclog}{(M,D)^\clog}
\newcommand{\tMclog}{(\tilde M,\tilde D)^\clog}
\newcommand{\Mmlog}{(M,D)^\mlog}
\newcommand{\Mmlogi}{(M,D_i)^\mlog}
\newcommand{\tMmlog}{(\tilde M,\tilde D)^\mlog}
\newcommand{\MlogP}{(M',D')^\log}
\newcommand{\MclogP}{(M',D')^\clog}
\newcommand{\MmlogP}{(M',D')^\mlog}
\newcommand{\Mcpl}{(M,D)^\cpl}
\newcommand{\Mtop}{(M,D)^\top}
\newcommand{\Mmot}{(M,D)^\mot}

\newcommand{\Mcplloc}{(M,D)^\cpl_{x_0}}
\newcommand{\Mtoploc}{(M,D)^\top_{x_0}}
\newcommand{\Mmotloc}{(M,D)^\mot_{x_0}}
\newcommand{\Mmlogloc}{(M,D)^\mlog_{x_0}}
\newcommand{\Mlogloc}{(M,D)^\log_{x_0}}
\newcommand{\Mclogloc}{(M,D)^\clog_{x_0}}
\newcommand{\Mcpli}{(M,D_i)^\cpl}
\newcommand{\Mtopi}{(M,D_i)^\top}
\newcommand{\Mmoti}{(M,D_i)^\mot}
\newcommand{\tMcpl}{(\tilde M,\tilde D)^\cpl}
\newcommand{\tMtop}{(\tilde M,\tilde D)^\top}
\newcommand{\tMmot}{(\tilde M,\tilde D)^\mot}

\newcommand{\MD}{(M,D)}
\newcommand{\tMD}{(\tilde M,\tilde D)}
\newcommand{\Sf}{\mathcal S_f}
\newcommand{\tN}{\tilde N}

\newcommand{\LLL}{\mathcal E}
\newcommand{\LLC}{{\mathcal E_1}}
\newcommand{\ta}{N}
\newcommand{\IsoTo}{\xrightarrow{
   \,\smash{\raisebox{-0.65ex}{\ensuremath{\scriptstyle\sim}}}\,}}

\newcommand{\sign}{\operatorname{sign}}
\newcommand{\restrD}{_{|D}}
\newcommand{\restrDi}{_{|D_i}}
\newcommand{\restrDo}{_{|{\Do_J}}}

\newcounter{subtag}

\newcommand{\inte}{\mathrm{int}}

\tikzcdset{arrow style=math font}

\begin{document}
\maketitle

\blfootnote{This research originated during the second and the third authors stay at the Mathematisches Forchung Institut Oberwolfach within the Research in Pairs Program in April-May 2018. During the preparation of this work, the first author was a postdoctoral fellow at the University of Toronto, supported by the NSERC Discovery Grant RGPIN-2017-06537, the second author was supported by the ANR project Défigéo (ANR-15-CE40-0008), and the third author was supported by the ANR project LISA (ANR-17-CE40-0023-03). \\

The authors would like to thank the anonymous referee for a thorough reading of the paper and several remarks, comments and suggestions, which helped to improve its readability. \\


}

\begin{abstract}
We compare the topological Milnor fibration and the motivic Milnor fibre of a regular complex function with only normal crossing singularities by introducing their common extension: the complete Milnor fibration. We give two equivalent constructions: the first one extending the classical Kato--Nakayama log-space, and the second one, more geometric, based on the real oriented multigraph construction, a version of the real oriented deformation to the normal cone. As an application, we recover A'Campo's model of the topological Milnor fibration, by quotienting the motivic Milnor fibration with suitable powers of $\Rp$, and show that it determines the classical motivic Milnor fibre.

We also give precise formulae expressing how the introduced objects change under blowings-up. As an application, we show that the motivic Milnor fibre is well-defined as an element of a suitable Grothendieck ring without requiring that the Lefschetz motive be invertible.
\end{abstract}

\newpage
\tableofcontents

\newpage

\section{Introduction}

\subsection{Context}
To a polynomial function germ $f:(\C^n,0)\to(\C,0)$, J. Milnor in \cite{Mil68} associates a smooth locally trivial fibration that, for some particular $f$, allows him to show that the link of $f^{-1}(0)$ at the origin is an exotic sphere, i.e. a manifold homeomorphic but not diffeomorphic to the unit sphere. This fibration, now referred to as the Milnor fibration, is defined by $f/|f|:S_\varepsilon \setminus f^{-1}(0)\to S^1$, where $S_\varepsilon$ is the sphere centred at the origin in $\C^n$ of a sufficiently small radius $\varepsilon >0$. As Lê D. T. shows in \cite{Le77}, if $\varepsilon>0$ is small enough and if $\delta>0$ is much smaller than $\varepsilon$ then the Milnor fibration is smoothly equivalent to the \emph{monodromy fibration} $$f_{|B_\varepsilon\cap f^{-1}(D^*_\delta)}:B_\varepsilon\cap f^{-1}(D^*_\delta)\to D^*_\delta,$$
where $B_\varepsilon$ and $D^*_\delta$ denote the open ball centred at the origin in $\C ^n$ of radius $\varepsilon$ and the punctured open disc centred at the origin in $\C$ of radius $\delta$, respectively. As it is now customary, by the topological Milnor fibration and the Milnor fibre we mean any of these fibrations and its fibre, up to smooth diffeomorphism.

The Milnor fibre and the Milnor fibration turned out to be a fundamental source of invariants of complex analytic hypersurface germs and an essential tool in the study of hypersurface singularities. The Milnor fibration can be investigated via a variety of methods: Morse theory, resolution of singularities, sheaf theory, oscillatory integrals, mixed Hodge theory, D-modules and microlocal methods, among others (we refer to \cite{Hand1}, \cite{Hand3}, \cite{Dim92}, \cite{AGV2} and the references therein).

Note that neither the total space of the Milnor fibration nor the Milnor fibre is an algebraic variety. In \cite{DL98}, J. Denef and F. Loeser introduce the motivic Milnor fibre which lies in the localisation of some equivariant Grothendieck ring of algebraic varieties. In what follows, we use an equivalent setting due to Guibert--Loeser--Merle \cite{GLM06}; see \S\ref{ss:mMf}. The motivic Milnor fibre of $f$ at the origin is defined by means of arc spaces and expressed in terms of a resolution of singularities, as a formal limit of a rational power series called the motivic zeta function. The localisation, which simply consists in making the Lefschetz motive invertible, is necessary to use motivic integration in order to prove that the motivic Milnor fibre formula doesn't depend on the choice of a resolution. The latter is a motivic incarnation of the Milnor fibration, even if it is called the motivic Milnor fibre, in the sense that the common known realisations (additive invariants) of both the motivic Milnor fibre and the topological Milnor fibration coincide \cite{DL98}, \cite{GLM06}.

In this paper we address the following question: is there a geometric object encoding both topological Milnor fibration and motivic Milnor fibre? In this direction, J. Nicaise and J. Sebag construct in \cite{NS} a rigid analytic Milnor fibre by means of non-Archimedean geometry whose étale cohomology is related to the topological Milnor fibre (together with the action of monodromy). 
Using Hrushovski--Kazhdan version of motivic integration, 
E. Hrushovski and F. Loeser introduce in \cite{HL} a definable Milnor fibre from which they recover the motivic Milnor fibre and the Lefschetz numbers of the monodromy, without using the resolution of singularities. Then J. Nicaise and S. Payne obtain in \cite{NP} a similar result and show that there is no need to localise the Grothendieck ring with respect to the Lefschetz motive for the motivic Milnor fibre to be well-defined. See also \cite{For19}.

We propose a geometric framework that includes both topological Milnor fibration and motivic Milnor fibre. For this purpose, we give two different though equivalent constructions. The first one generalizes the log-geometric construction pioneered by K. Kato and C. Nakayama in \cite{KN99} and applied here to the topological Milnor fibration. The second one is based on real oriented blowings-up, following N. A'Campo \cite{AC75}, together with a version of the deformation to the normal cone that is used in \cite{GLM06}. Assuming that $f:M\to\C$, defined globally, has only normal crossing singularities along $D\coloneq f^{-1}(0)$, we construct a space $\Mcpl$ over $M$, stratified over the canonical stratification on $M$ induced by the irreducible components of $D$. The superscript $\cpl$ stands for \emph{complete} since it contains our model of the motivic Milnor fibration space, denoted by $\Mmot$, and since its boundary $\Mtop$ coincides with A'Campo's first model of the topological Milnor fibration space. To be more precise, we obtain these models of Milnor fibration only after taking restrictions of these three spaces to $D$.

The use of logarithmic geometry to study the topological Milnor fibration and fibre is already present, implicitely, in \cite{AC75}. In the context of the weight and Hodge filtrations, it appears in \cite{Kaw02}. It has been used recently in \cite{Cau16}, \cite{CPS22}, and, again implicitely, in \cite{BP22}. We are not aware of the previous use of log-geometry to study the motivic Milnor fibre.

We have to assume that $f$ is a globally defined algebraic function having only normal-crossing singularities because we need a local combinatorial structure. It is likely that our method can be generalized to the toroidal case, we will study it in forthcoming papers. In recent years, there have been many efforts to give computational expressions of the motivic Milnor fibre that do not rely on resolution of singularities. We do not know at the moment how to put our constructions in such a framework. Instead, we provide a complete description of the effect of a blowing-up on each the three spaces $\Mcpl$, $\Mmot$, and $\Mtop$, that we find surprisingly simple and elegant.

\subsection{Set-up}\label{ss:setup}
Let $M$ be a nonsingular complex algebraic variety and let $f:(M,D)\to(\C,0)$ be a regular function. We assume that $D\coloneq f^{-1}(0)$ is a divisor with simple normal crossings, that is $D=\bigcup_{i\in I}D_i$ where the $D_i$'s are nonsingular hypersurfaces meeting transversally.

For $i\in I$, let $p_i:L_i\to M$ be the algebraic line bundle associated with $D_i$ together with $s_i$ be a regular section of $L_i$ such that $D_i$ is defined by $s_i$ as a subvariety (in particular $D_i$ is reduced and $s_i$ is transverse to the zero section), see for instance \cite[\S1.1]{GH}. Then we can write
\begin{equation}\label{eqn:globalformula}
f(x) = u(x) \prod_{i\in I} s_i(x) ^{N_i}
\end{equation}
where $u(x)$ is a nowhere vanishing regular function and $(N_i)_{i\in I}\in\mathbb N^I$. 
For $J\subset I$, we set $$D_J=\bigcap_{j\in J}D_j\quad\quad\text{ and }\quad\quad\Do_J=D_J\setminus\bigcup_{i\notin J}D_i$$
Note that $D_{\varnothing}=M$ and that $\Do_{\varnothing}=M\setminus D$, therefore $\left\{\Do_J\right\}_{J\subset I}$ and $\left\{\Do_J\right\}_{\varnothing\neq J\subset I}$ are stratifications of $M$ and $D$, respectively.

The use of global line bundles and global sections replaces local computations so that we don't have to check that our constructions don't depend on the choice of local coordinates. Though the sections $s_i$ are not uniquely defined, the log-geometric approach gives canonical definitions of our constructions which don't depend on the choice of these sections.

\subsection{Outline of the paper and main results}
First, in Section \ref{sec:logtheoric}, we give simple sheaf-theoretic definitions of $\Mclog$, $\Mmlog$, and $\Mlog$ which are of log-geometric nature, 
the space $\Mlog$ being the classical Kato--Nakayama log-space \cite{KN99}. Recall that \emph{the divisorial log structure on $M$ associated to $D$} is the sheaf of multiplicative monoids defined for an open set $U$ by
\begin{equation}\label{eq:MU}\M(U)\coloneq\left\{f\in\OO_M(U)\,:\,f_{|U\cap(M\setminus D)}\text{ is invertible}\right\}.\end{equation}
In particular for every $x\in M$ we have the following exact sequence of monoids
\begin{align}\label{eq:logexactsequence}
0\to \OO_x^* \to \M_x\to \N^{J(x)} \to 0,
\end{align}
where $J(x) \coloneq \{i\in I ; x\in D_i\}$ and the map $\M_x\to \N^{J(x)}$ sends the germ of $f=u\prod s_i^{N_i}$ at $x$  to $\{N_i : i\in J(x)\}$. The Kato--Nakayama log-space $\Mlog$ is defined stalkwise as the space of monoid morphisms from $\M_x$ to $S^1$. The spaces $\Mmlog$ and $\Mclog$ are defined in the same way with $S^1$ replaced by $\C^*$ and $S^1\times (0,\infty]$, respectively; see Section \ref{sec:logtheoric}.

 These spaces are equipped with natural stratifications coming from the canonical stratification of $D=\bigsqcup_{\varnothing\neq J\subset I}\Do_J$, endowing $\Mmlog_{|\Do_J}$ with a structure of $(\C^*)^{J}$-torsor, and $\Mlog_{|\Do_J}$ with a structure of $(S^1)^{J}$-torsor. This construction is functorial and, after taking the restriction to $D$, this functoriality induces the following commutative diagram
\begin{equation}\label{eqn:NiceDiagram}
\begin{tikzcd}
\Mmlog\restrD \arrow[r,hook]  \arrow[d,"f^\mlog\restrD"]  &  \Mclog\restrD \arrow[d,"f^\clog\restrD"] \arrow[r,"pr"]
& \Mlog\restrD \arrow[d,"f^\log\restrD"] \\
(\C,0)^\mlog_{|0} \arrow[d,phantom, sloped, "="] & (\C,0)^\clog_{|0} \arrow[d,phantom, sloped, "="] & (\C,0)^\log_{|0} \arrow[d,phantom, sloped, "="] \\[-15pt]
\C^*\arrow[r,hook] & \Ri\times S^1 \arrow[r]& S^1
\end{tikzcd}
\end{equation}
whose interpretation is the core of the paper.

\begin{thmx}\label{thm:A}
The map $f^\log\restrD$ coincides with A'Campo's first model of the topological Milnor fibration. The reduction of $f^\mlog\restrD$ gives the motivic Milnor fibre $\Sf$.  We obtain $f^\log\restrD$ by dividing  $f^\mlog\restrD$ over each stratum $\Do_J$ by $(\Rp)^{J}$ in the source and $\Rp$ in the target.
\end{thmx}

We refer the reader to the beginning of Section \ref{sec:acampo} for a more detailed explanation concerning A'Campo's models of the topological Milnor fibration.

The reduction of $f^\mlog\restrD$, defined in Definition \ref{def:reduction}, associates to $f^\mlog\restrD$ an element in the Grothendieck ring of equivariant algebraic varieties. 

The proof of Theorem \ref{thm:A} relies on geometric versions of these log-spaces. For that purpose we define $\Mcpl$, $\Mmot$, and $\Mtop$, which correpond to $\Mclog$, $\Mmlog$, and $\Mlog$, respectively. The construction of these spaces is more geometric albeit more technical. It is based on MacPherson's Grassmannian graph construction \cite{Mac74}, a generalisation of the deformation to the normal cone. In Section \ref{sec:multigraph}, we use a multigraph version of it; that is, we apply it simultaneously to the sections $s_i$. This suffices to define $\Mmot$ and to show that the reduction of $f^\mlog$ to the Grothendieck ring gives the motivic Milnor fibre. Up to this point, our approach is similar to that of \cite{GLM06}.

In Section \ref{sec:cpl}, we introduce the \emph{real oriented multigraph construction} and define formally $\Mcpl$, $\Mmot$, and $\Mtop$, see Definition \ref{defi:Mtop}. Only $\Mmot$ is a complex algebraic variety; it is defined in terms of algebraic line bundles over the canonical stratification of $M$ given by $D$. The space $\Mtop$ is obtained by dividing these bundles by the action of $\Rp$ and thus becomes a semialgebraic set. Nonetheless, thanks to its stratified structure, $\Mtop$ still carries a lot of information such as the classical log-space; see also Theorem \ref{thm:B} below. Note that these three spaces, unlike the log-spaces $\Mclog$, $\Mmlog$, and $\Mlog$, depend on the choice of the sections $s_i$, though this dependence is marginal. Moreover, as we show in \S\ref{ss:loggeo}, each of the spaces $\Mcpl$, $\Mmot$, and $\Mtop$ is canonically isomorphic to $\Mclog$, $\Mmlog$, and $\Mlog$, respectively. The proof is geometric and is based on Lemma \ref{lem:log-and-normal} and Corollary \ref{cor:log-and-normal} that show that "the essential part" of each of the three log-spaces can be expressed canonically in terms of normal bundles.

Using this description, we prove Theorem \ref{thm:A} in \S\ref{ss:loggeo}. The geometric intuition behind these constructions is that one may interpret $\Mcpl\restrD$ as an infinitesimal punctured neighbourhood of the zero set $D$ of $f$ whose boundary is $\Mtop\restrD$. It comes with a map, $\sign f: \Mcpl\restrD \to S^1$, whose restriction to $\Mtop\restrD$ gives A'Campo's first model of the topological Milnor fibration. Moreover, $\Mcpl$ also contains the algebraic variety $\Mmot$. The monodromy fibration, originally given by $f$ itself restricted to a punctured neighbourhood of $D$, induces stratumwise maps $f_J:\Mmot\restrDo\to\mathbb C^*$, for $\varnothing\neq J\subset I$, involved in the definition of the motivic Milnor fibre $\Sf$. \\

It is known that there exist topologically equivalent functions whose motivic Milnor fibres don't coincide, and non topologically equivalent functions with the same motivic Milnor fibre, see \cite{CNR}. However, taking advantage of the present stratified setting, the motive $\Sf$ is entirely determined by A'Campo's first model of the topological Milnor fibration.

\begin{thmx} \label{thm:B}
The motivic Milnor fibre $\Sf$ is determined by the stratified topological Milnor fibration $\sign f:\Mtop\restrD\to S^1$.
\end{thmx}

The proof of Theorem \ref{thm:B} is given in \S\ref{ss:strat}.

The following diagram roughly summarizes the situation, the diagonal arrow is defined by dividing stratumwise by powers of $\Rp$: $$\begin{tikzcd}\Mcpl\restrD  & \Mmot\restrD  \arrow[ld] \arrow[l,hook'] \\ \Mtop\restrD \arrow[u,hook] &
\end{tikzcd}.$$
The diagonal map is the (stratumwise) geometric quotient by $(\Rp)^{|J|}$ in the source space and $\Rp$ in the target. This allows us to give an interpretation of the minus sign in front of the sum and to the coefficients $(-1)^{|J|}$ in the definition of the motivic Milnor fibre $\Sf$, see Remarks \ref{rmk:speculations} and \ref{rmk:speculation2} :
$$\Sf \coloneq -\sum_{\varnothing\neq J\subset I}(-1)^{|J|}\left[(\pi_J,f_J):\C^*\acts L_J^\star\to D\times\C^*\right]\in\KDC.$$

In Section \ref{sec:acampo}, devoted to the topological Milnor fibration, we show that A'Campo's \cite{AC75} second model, i.e. the model of the topological Milnor fibration used to compute the zeta function of the monodromy, can be obtained from $\sign f: \Mcpl \to S^1$.

\begin{thmx}\label{thm:C}
After division by the diagonal action of $\Rp$, the map $\sign f: \Mcpl \to S^1$ induces a map $$\left(\Mcpl\setminus\Mtop\right)/\Rp \to S^1$$ that coincides with the A'Campo's  second  model of the topological Milnor fibration.
\end{thmx}

Equivalently, the second model is obtained from $$\left(\Mclog\setminus\Mlog\right) \to \left( (\C,0)^{\clog} \setminus (\C,0)^{\log}\right) = \Rp \times S^1$$ by dividing both the source and the target by the action of $\Rp$. Our description is more precise as we give a natural embedding of A'Campo's second  model in $\Mcpl$, see Propositions \ref{prop:AC1} and \ref{prop:AC2}. Note that our construction is, in principle, similar to the recent construction of the topological Milnor fibre equipped with monodromy given in \cite{BP22}; we do not use a partition of unity, instead the passage from local to global relies on line bundles and their sections. Nonetheless, we do not go as far as in \cite{BP22}, where a symplectic model of the monodromy is constructed.

In Section \ref{sec:bl}, we study how the spaces $\Mclog$, $\Mmlog$ and $\Mlog$ change under the blowing-up of $M$ along a centre that is assumed to be included in $D$ and in normal crossings with $D$. This is given by several precise formulae expressed in terms of the normal bundles of the centres, the strata, and the exceptional divisor. They rely, again, on Lemma \ref{lem:log-and-normal} and Corollary \ref{cor:log-and-normal}.

The formulae given in Section \ref{sec:bl} are applied in Section \ref{sec:localcase}, where we return to the original problem of relating the local topological Milnor fibration and the motivic Milnor fibre. By resolution of singularities, any algebraic  $f: (\C^n,0)\to (\C ,0)$ can be resolved to $f:(M,D)\to(\C,0)$ as in the set-up.  To a singularity resolved thusly, we may associate the complete Milnor fibration space $\Mcpl$ together with the subspaces $\Mmot$ and $\Mtop$.  By the weak factorisation theorem \cite{Wlo03},  the formulae given in Section \ref{sec:bl} show how such Milnor fibration spaces obtained for different resolutions are related. In particular, we show that the motivic Milnor fibre $\Sf$ is well-defined as an element in the Grothendieck ring $\KDC$ and not merely in its localisation, as already noted in \cite{NP} and \cite[Proposition 5.14]{For19}. We conjecture that the equivalences given in Section \ref{sec:bl} are much stronger and show the existence of the motivic Milnor fibre in a much finer Grothendieck ring.

\subsection{Notation}\label{ss:notation}
Given a vector bundle $E$, we use the notation $E^*$ for $E$ with its zero section removed and $\Sph(E)$ for the  associated sphere bundle; i.e. $\Sph(E)$ is fibrewise defined by $\Sph(E)_x=\left(E_x\setminus\{0\}\right)/\Rp$.
If $E$ is equipped with a Hermitian metric $<\cdot,\cdot>:E\otimes E\to\C$, then $\Sph(E)$ can be identified with the set of $v\in E$ satisfying $<v,v>=1$.

Given a family of fibre bundles $(\xi_i:E_i\to X)_{i\in I}$, we denote their fibre product by $$E_I\coloneq\fprod{X}{i\in I}E_i=\left\{(y_i)_{i\in I}\,:\,\xi_i(y_i)=\xi_j(y_j)\right\}.$$
When the $\xi_i$ are vector bundles and $I$ is finite, we simply obtain the usual direct sum construction. Nonetheless, we have to work with more general fibre bundles, hence this notation.

Given a family of vector bundles $(\xi_i:E_i\to X)_{i\in I}$, we denote by $E_I^\star$ the fibre product of the $E_i^*$, and by $\Sph^\star(E_I)$ the fibre product of the $\Sph(E_i)$, namely
$$E_I^\star \coloneq \fprod{X}{i\in I}E_i^*\quad\quad\quad\text{and}\quad\quad\quad\Sph^\star(E_I) \coloneq \fprod{X}{i\in I}\Sph(E_i).$$
These bundles are not to be mistaken with $E_I^*$ and $\Sph(E_I)$, respectively, which play also a role in section \ref{sec:bl}.

For a vector $v$ of a line bundle $L$, we set $r(v)\coloneq v\mod S^1$,  $L_{\ge0}\coloneq\left(L \mod S^1\right)$ and $L^*_{ >0}\coloneq\left(L^*\mod S^1\right)$, and, if additionally $v\ne 0$, $\theta(v)\coloneq v\mod\Rp\in \Sph(L)$. If $L$ is a line bundle over $X$, then $v\to (r(v), \theta(v))$ is a bijection between $L^* $ and $L^*_{>0}\times_X \Sph(L)$. We denote the inverse of this bijection by $(r,\theta) \to v=r\theta$, and we sometimes extend it to $r=0$, in which case $v=r\theta$ is the zero vector.

\section{Log constructions of the Milnor fibrations}\label{sec:logtheoric}
Let $M$ be a nonsingular complex algebraic variety and $f:(M,D)\to(\C,0)$ be a regular function such that $D\coloneq f^{-1}(0)$ is a divisor with simple normal crossings; i.e. $D=\bigcup_{i\in I}D_i$ with $D_i$ nonsingular hypersurfaces meeting transversally.

In this section, we recall the classical Kato--Nakayama log-space $\Mlog$ and introduce two extensions.  The classical Kato--Nakayama log-space $\Mlog$ was introduced in \cite{KN99} (we refer the reader to \cite{Ogu18} for a detailed exposition on log geometry). We define below a \emph{complete log-space} denoted by $\Mclog$.  It contains both $\Mlog$ and the \emph{algebraic log-space} $\Mmlog$ also defined below. Geometrically, these log-spaces extend the normal bundles of the components of $D$ in $M$; see Lemma \ref{lem:log-and-normal}.

\subsection{Log-spaces}
Let $\OO_M$ be the structure sheaf of $M$. Then a \emph{prelog structure} on $M$ is given by $\M$ a sheaf of monoids over $M$ together with a morphism of sheaves of monoids $\alpha:\M\to\OO_M$ where $\OO_M$ is seen as a sheaf of monoids for the multiplication. We say that a prelog structure is a \emph{log structure} if, furthermore, $\alpha:\alpha^{-1}\left(\OO_M^*\right)\to\OO_M^*$ is an isomorphism (so that $\OO_M^*$ can be seen as a subsheaf of $\M$). \\

Following \cite{KN99}, the classical \emph{Kato--Nakayama log-space} of $(M,D)$ is then defined by
\begin{align*}
\Mlog\coloneq\left\{(x,\varphi)\,:\;x\in M,\,\varphi\in\Hom_{\mon}(\M_x, S^1),\,\forall g\in\OO_x^*,\,\varphi(g)=\frac{g(x)}{|g(x)|}\right\} ,
\end{align*}
where $\M$ is the divisorial log structure on $M$ associated to $D$, see \eqref{eq:MU}, and, the structure of monoid on $S^1$ is just its (multiplicative) group structure.

The \emph{complete log-space} $\Mclog$ and the \emph{algebraic log-space} $\Mmlog$ are defined by
\begin{align*}
& \Mclog\coloneq \\ & \left\{(x,\varphi,\psi):(x,\varphi)\in\Mlog,\,\psi\in\Hom_{\mon}\left(\M_x,\left(\Ri,\cdot\right)\right),\forall g\in\OO_x^*,\,\psi(g)
= |g(x)| \right\} ,
\end{align*}
and
\begin{align*}
\Mmlog\coloneq\left\{(x,\Phi)\,:\,x\in M,\,\Phi\in\Hom_{\mon}(\M_x,\C^*),\,\forall g\in\OO_x^*,\,\Phi(g)=g(x)\right\},
\end{align*}
respectively, where $\Ri$ and $\C^*$ are equipped with their monoid structures induced by multiplication.

Over $M\setminus D$, all three spaces $\Mlog$, $\Mmlog$ and $\Mclog$ are isomorphic to $M\setminus D$ since $\M_x=\OO_x^*$ for $x\in M\setminus D$.

The following diagram summarizes the situation at the level of monoids:
\begin{equation}\label{eq:mono}
\begin{tikzcd}
\C^* \arrow[r,hook] & \C^*\cup\{\infty\} & \\
\Rp\times S^1 \arrow[r,hook] \arrow[u,"\simeq"] \arrow[dr, "\pr_{S^1}"']  &  \Ri \times S^1 \arrow[d,"\pr_{S^1}"] \arrow[u,"\pi"] \arrow[r,hookleftarrow,dashed] & \{\infty\}\times S^1  \\
& S^1 & &
\end{tikzcd}.
\end{equation}
The plain arrows are morphisms of monoids; the dotted arrow is a morphism of semigroups (recall that a monoid is a semigroup equipped with an identity element\footnote{Note that the usual convention in toric geometry consists in using the term \emph{semigroup} when referring to a \emph{monoid}.}); and $\pi:\Ri\times S^1 \to\C^*\cup\{\infty\}$ denotes the compactification of $\C^*$ by a circle at infinity, i.e. $\pi(r,e^{i\theta})=re^{i\theta}$ (see \S\ref{ss:rog}).

The natural isomorphism $\Hom_{\mon}(\M_x,\Rp)\times\Hom_{\mon}(\M_x,S^1)\to\Hom_{\mon}(\M_x,\C^*)$ allows us to identify $\Mmlog$ with a subspace of $\Mclog$. Similarly, the classical log-space $\Mlog$ can be identified with the subset of $(\varphi,\psi)\in\Mclog$ satisfying $\psi(g)=\infty$ for all $g$ vanishing at $x$.

Note also that $\Mclog$ projects canonically onto $\Mlog$ by forgetting the morphism $\psi$. Similarly $\Mmlog$ projects canonically onto $\Mlog$ by dividing by $\Rp$, i.e. taking the quotient by the action of $\Rp$ on $\C^*$.

The following commutative diagram, induced by \eqref{eq:mono}, summarizes the situation at the level of the log-spaces.
$$
\begin{tikzcd}
\Mmlog \arrow[r,hook]  \arrow[dr, "/\Rp"']  &  \Mclog  \arrow[d,"pr"]  \arrow[r,hookleftarrow,dashed] & \Mlog \arrow[dl,equal] \\
& \Mlog & &
\end{tikzcd}
$$

\begin{ex}
For $M=\C$ and $D=0$, we have above the origin $$(\C,0)^\mlog_{|0}\sqcup(\C,0)^\log_{|0}=\C^*\sqcup S^1=(\Rp\times S^1)\sqcup(\{\infty\}\times S^1)=\Ri\times S^1=(\C,0)^\clog_{|0}.$$
The diagram \eqref{eq:mono} may be rewritten as
$$
\begin{tikzcd}
\C^* \arrow[r,hook] & \C^*\cup\{\infty\} & \\
(\C,0)^\mlog_{|0} \arrow[r,hook] \arrow[u,"\simeq"] \arrow[dr, "/\Rp"']  &   (\C,0)^\clog_{|0} \arrow[d,"pr"] \arrow[u,"\pi"] \arrow[r,hookleftarrow,dashed] &  (\C,0)^\log_{|0} \\
& (\C,0)^\log_{|0} & &
\end{tikzcd}
$$
\end{ex}

\begin{ex}\label{ex:log}
For $M=\C^2$ and $D : x_1x_2=0$, the fibre of $\Mclog$ above $x\in D\setminus\{O\}$ is $\Mclog_{|x}=(\C,0)^\clog$ whereas above the origin we have \begin{align*}\Mclog_{|O}&=\left(\Ri\times S^1\right)^2\\&=\left((\{\infty\}\times S^1)\sqcup(\Rp\times S^1)\right)^2\\&=(S^1\sqcup\C^*)^2\\&=(S^1\times S^1)\sqcup(\C^*\times\C^*)\sqcup(\C^*\times S^1)\sqcup(S^1\times\C^*)\\&=\Mlog_{|O}\sqcup\Mmlog_{|O}\sqcup(\C^*\times S^1)\sqcup(S^1\times\C^*)\end{align*}
Note that $\Mlog_{|O} \sqcup \Mmlog_{|O}$ is strictly included in $\Mclog_{|O}$: its complement is equal to the set of elements $(O,\phi,\psi)\in\Mclog_{|O}$ satisfying $\psi(x_1x_2)=\infty$ and $(\psi(x_1),\psi(x_2))\neq(\infty,\infty)$. We call this complement the \emph{mixed part}; see also Examples \ref{ex:ex0} and \ref{ex:bu}.
\end{ex}

\begin{rmk}
Although equivalent, it would be probably more natural to replace the monoid $(0,\infty]$ by $[0,\infty)$ in our definition of the complete log-space. In our second approach, based on MacPherson's graph construction, see Section \ref{sec:multigraph}, the fibres over the divisor $D$ are sent to infinity and not to zero; we have thus made the choice to follow the latter more geometric approach when defining the complete log-space.
\end{rmk}

\begin{rmk}\label{rmk:cau}
The complete log-space $\Mclog$ differs from the extended log-space defined by Cauwbergs in \cite{Cau16} as:
$$\MD^\mathrm{ext}\coloneq
\left\{(x,\varphi,\psi)\,:\,(x,\varphi)\in\Mlog,\,
\psi\in\Hom_{\mon}\left(\M_x,(\R_{\ge 0}, + )\right)\right\} .
$$
Both our construction and Cauwbergs' provide a functorial construction of A'Campo's first model for the topological Milnor fibration together with its monodromy (see Section \ref{sec:acampo}); the main difference is that our complete log-space $\Mclog$ contains $\Mmlog$.
\end{rmk}

\subsection{Functoriality}\label{ss:functor}
Let $F:M\to M'$ be a morphism between nonsingular algebraic varieties with simple normal crossings divisors satisfying $F^{-1}(D')\subset D$. Then $F$ gives rise to a morphism of sheaves $F^*:F^{-1}\M'\to\M$ (with $\M$ and $\M'$ defined in \eqref{eq:MU}) inducing the following morphisms of monoids at the level of stalks $$F_x^*:\begin{array}{ccc}\M_{F(x)}'&\to&\M_x\\f&\mapsto&f\circ F\end{array}.$$
Therefore, $F$ induces the morphisms
\begin{align*}
& F^{\log}:\Mlog\to\MlogP, \\
& F^{\clog}:\Mclog\to\MclogP, \\
& F^{\mlog}:\Mmlog\to\MmlogP.
\end{align*}
We apply this observation to the function $f:(M,D)\to(\C,0)$ from the set-up, so that we get the following diagram
\begin{equation}\label{eqn:GMFalt}
\begin{tikzcd}
\Mmlog \arrow[r,hook]  \arrow[d,"f^\mlog"]  &  \Mclog \arrow[d,"f^\clog"] \arrow[r,"pr"]  & \Mlog \arrow[d,"f^\log"] \\
(\C,0)^{\mlog}\arrow[r,hook] & (\C,0)^{\clog} \arrow[r]& (\C,0)^{\log}
\end{tikzcd}
\end{equation}
Restricting to $D$, we obtain diagram \eqref{eqn:NiceDiagram}.
We call $f^\mlog\restrD$ the \emph{log monodromy fibration} analogously to the monodromy fibration induced by $f$ on $f^{-1}(\{z\in\C\,:\,0<|z|\le\varepsilon\})$.

Composing $f^\clog\restrD$ with $(\C,0)^\clog_{|0}\simeq \Ri\times S^1\xrightarrow{\pr_{S^1}}S^1$, we obtain the continuous map
\begin{equation}\label{eqn:GMF}
\sign f:\Mclog\restrD\to S^1
\end{equation}
that we call \emph{the log-complete Milnor fibration of $f$}.

\subsection{Algebraic log-space and normal bundles}\label{ss:normalls}
In what follows, we describe the algebraic log-space $\Mmlog$ in terms of the normal bundles of the irreducible components of $D$. This description will be useful in the proof of Theorem \ref{thm:logspaces} and in Section \ref{sec:bl}.
Let $N_i$ be the normal bundle of $D_i$ in $M$. Recall that it is defined as a quotient bundle by the exact sequence
$$0\to TD_i \to TM_{|D_i} \to N_i \to 0.$$
Then the following lemma allows us to identify canonically $\Mmlogi\restrDi$ and $N_i^*$.

\begin{lemma}\label{lem:log-and-normal}
Let $x\in D_i$.  Then there is a canonical bijection
\begin{align}\label{eqn:mlog-and-normal}
\Mmlogi_x \ni \Phi  \leftrightarrow v \in \left(N^*_i\right)_x
\end{align}
given by the relation $\Phi(g)=\partial_{v}g(x)$ where $g\in\OO_{x}$ is any generator of the ideal of $D_i$.
\end{lemma} 

\begin{proof}
If $\tilde g = ug$, where $u(x) \ne 0$, then  $\partial_{v}\tilde g(x)=u(x)\partial_{v}g(x)$ and similarly $\Phi(\tilde g)=u(x)\Phi(g)$. Thus the above formula does not depend on the choice of generator $g$.
\end{proof}

Following the notation from \S\ref{ss:notation}, we set $$N^\star_J{}_{|\Do_J}\coloneq \fprod{\Do_J}{i\in J}{N_i^*}_{|\Do_J}\quad\text{ and }\quad\Sph^\star ({N_J}_{|\Do_J})\coloneq\fprod{\Do_J}{i\in J}\Sph({N_i}_{|\Do_J}),$$ then we get the following corollary.

\begin{cor}\label{cor:log-and-normal}
For $J\subset I$, the identification \eqref{eqn:mlog-and-normal} allows us to identify canonically
\begin{equation}\label{eqn:log-and-normal}
\Mmlog_{|\Do_J} \simeq N^\star_J{}_{|\Do_J}\quad\text{ and }\quad \Mlog_{|\Do_J}\simeq\Sph^\star ({N_J}_{|\Do_J}).
\end{equation}
\end{cor}

\section{The multigraph construction and the motivic Milnor fibre}\label{sec:multigraph}
In this section, we compute the motivic Milnor fibre using a construction that we call the \emph{multigraph construction}, that is the fibre product of several MacPherson's graph constructions \cite{Mac74}.

\subsection{MacPherson's graph construction}\label{ss:McP}
Before introducing the graph construction, we briefly recall the related construction of the deformation to the normal cone. 
Let $X$ be a closed subspace of a nonsingular complex algebraic variety $M$.

The deformation to the normal cone $C_XM$ is denoted by $\PP_XM^\circ$ and is constructed as follows (see also \cite[Chapter 5]{Ful84}).

Let $\rho:\mathcal \PP_XM\to\PP^1$ be the composition of the blowing-up $\xi:\PP_XM\to M\times\PP^1$ of $M\times\PP^1$ along $X\times\{\infty\}$ with the projection $M\times\PP^1\to\PP^1$. Then $\rho^{-1}(\infty)$ is the union of two Cartier divisors $$\rho^{-1}(\infty)=\PP(C_XM\oplus\one)+\widetilde M.$$
Here $\PP(C_XM\oplus\one)$ is the exceptional divisor of $\xi$ and equals the projectivisation of the cone $C_XM\oplus\one$, where $\one$ denotes the trivial line bundle over $X$. Thus, $\PP(C_XM\oplus\one)=C_XM\sqcup\PP(C_XM)$.

The divisor $\widetilde M$ is the strict transform of $M\times\{\infty\}$ and is isomorphic to the blowing-up of $M$ along $X$. Thus the exceptional divisor of the latter blowing-up can be identified with the intersection $\PP(C_XM\oplus\one)\cap\widetilde M=\PP(C_XM)$.

In the description of $\PP_XM$ we consider $X\times\PP^1$ as canonically embedded in $\PP_XM$, i.e. $X\times\PP^1\subset\PP_XM$, identifying it with its strict transform in $\PP_XM$ (it is isomorphic to $X\times\PP^1$ because the blowing-up of $X\times\{\infty\}$ in $X\times\PP^1$ is the identity map). In particular $X\times\{\infty\}$ is the base of the cone $C_XM$.

We define $\PP_XM^\circ$ as the complement of $\widetilde M$ in $\PP_XM$ and $\rho^\circ$ as the restriction of $\rho$ to $\PP_XM^\circ$. Then $C_XM=(\rho^\circ)^{-1}(\infty)$ and $(\rho^\circ)^{-1}(\C)=\rho^{-1}(\C)$ is isomorphic to $M\times\C $.

This situation is summarized in the following diagram.
$$
\xymatrix{
X\times\PP^1 \ar@{^{(}->}[rr] \ar[rd] & & \PP_XM^\circ \ar[ld]^{\rho^\circ} \\
& \PP^1 &
}
$$

The deformation to the normal cone is a special case of MacPherson's Grassmannian graph construction which was introduced in \cite{Mac74}.

Let us recall this special case (MacPherson's construction generalizes to vector bundle morphisms $E_1\to E_2$ by considering Grassmanian spaces instead of projective spaces).

Assume that $X=s^{-1}(0)$ is the zero locus of a regular section $s$ of a vector bundle $\pi:E\to M$. Consider the embedding $$M\times\C\longhookrightarrow\PP(E\oplus\one)\times\PP_\mathbb C^1\stackrel{\displaystyle\rho}{\longrightarrow}\PP_\mathbb C^1,$$ given by $(y,\lambda)\to\left([\lambda s(y):1]_y,[\lambda:1]\right)$, and let $\rho$ be the projection onto the second factor. Here again, $\one$ denotes the trivial line bundle and $\PP(E\oplus\one)$ is the fibrewise projective compactification of the vector bundle $E\to M$. 
Then $\PP_XM$ is the closure of the image of $M\times\C$ in $\PP(E\oplus\one)\times\PP_\mathbb C^1$.

Suppose that $\operatorname{codim}(X)=\operatorname{rk}(E)$ 
(when $E$ is a line bundle -- the case considered in this paper -- it is enough to assume that the restriction of $s$ to every irreducible component of $M$ is not identically equal to zero). Then the fibre at infinity of the restriction of $\rho$ to $\PP_XM$ can be expressed as the following connected sum over $\PP(E_{|X})$
\begin{equation}\label{eqn:decomposition}\rho_{|\PP_XM}^{-1}(\infty)=\PP(E_{|X}\oplus\one)\bigcup_{\PP(E_{|X})}\tilde M,\end{equation}
where $\xi:\tilde M\coloneq\left\{\ell\in\PP(E),\,s(\pi(\ell))\in\ell\right\}\xrightarrow{\pi}M$
is the blowing-up of $M$ along $X$ and $\PP(E_{|X})$ is the exceptional divisor.

Indeed, first assume that $y\to y_0$ with $y_0\notin X$ and $\lambda\to\infty$; then $$[\lambda s(y):1]\to[s(y_0):0]\in\PP\left(E_{y_0}\right)\subset\tilde M\setminus\PP(E_{|X}).$$ Next, assume that $s(y)\to s(y_0)$ where $y_0\in X$ and $\lambda\to\infty$; then either $\lambda s(y)\to\infty$ and \linebreak[3]$[\lambda s(y):1]\to[1:0]=\infty\in\PP\left(E_{y_0}\right)\subset\PP(E_{|X})$, or $\lambda s(y)\to w\neq\infty$ and then $[\lambda s(y):1]\to[w:1]\in E_{y_0}\subset E_{|X}$.

We only need a special case when $E$ is a line bundle and $s$ is not identically equal to zero on each irreducible component of $M$. Then $C_XM=E_{|X}$, $\PP(C_XM)$ is isomorphic to $X$ and the blowing-up of $M$ along $X$ is the identity (this will not be the case when we shall replace blowings-up by real oriented blowings-up in the real oriented multigraph construction). We simply use the notation $\PP M$ instead of $\PP_XM$.

The following picture, inspired by \cite[Remark 5.1.1]{Ful84}, is illustrative. When $\lambda$ goes to infinity, we obtain the cone $\PP(E_{|X}\oplus\one)=\PP(E_{|X})\sqcup E_{|X}$ (the dashed vertical line) and the blowing-up $\tilde M$ of $M$ along $X$ (included in $\PP(E)$ which is the part at infinity represented by the upper horizontal line) intersecting along the exceptional divisor $\PP(E_{|X})$ (the dot in the middle of the upper horizontal line).
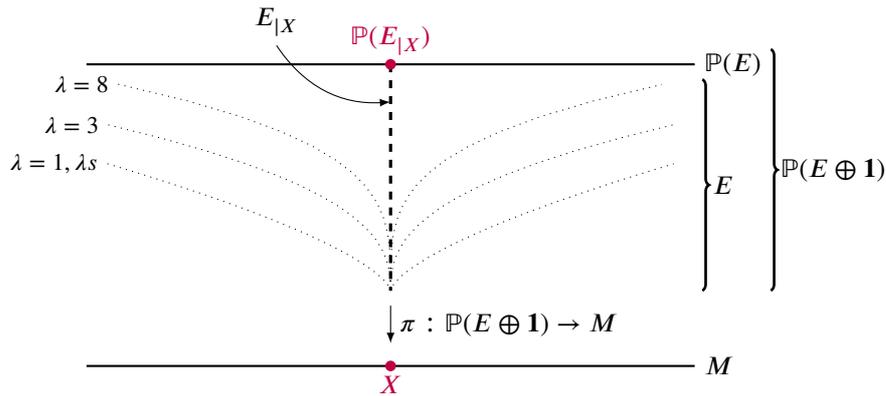
\begin{center}
\tikzset{mylabel/.style  args={at #1 #2  with #3}{
    postaction={decorate,
    decoration={
      markings,
      mark= at position #1
      with  \node [#2] {#3};
 } } } }
\begin{tikzpicture}
\draw[very thick,dashed] (0,3) -- (0,0);
\draw[latex-] (0,2.5) to [bend left] (-1.5,3.25) node[above] {$E_{|X}$};
\draw[thick] (-4,3) -- (4,3) node[right] {$\PP(E)$};
\fill[purple] (0,3) circle (2pt) node[above] {$\PP(E_{|X})$};
\draw[-latex] (0,-.2) -- (0,-.7) node[pos=.5,right] {$\pi:\PP(E\oplus\one)\to M$};
\draw[thick] (-4,-1) -- (4,-1) node[right] {$M$};
\fill[purple] (0,-1) circle (2pt) node[below] {$X$};
\draw[decorate,decoration={brace},line width=1pt] (4.1,2.8) -- (4.1,0) node[pos=.5,right] {$E$};
\draw[decorate,decoration={brace},line width=1pt] (5,3.2) -- (5,0) node[pos=.5,right] {$\PP(E\oplus\one)$};
\draw[dotted,domain=-1.55:1.55,samples=50,mylabel=at 0 left with {\small$\lambda=1,\lambda s$}] plot ({(\x)^3}, {.7*(abs(\x))^2});
\draw[dotted,domain=-1.3:1.3,samples=50,mylabel=at 0 left with {\small$\lambda=3$}] plot ({(\x)^5}, {1.3*(abs(\x))^2});
\draw[dotted,domain=-1.2:1.2,samples=50,mylabel=at 0 left with {\small$\lambda=8$}] plot ({(\x)^7}, {1.9*(abs(\x))^2});
\end{tikzpicture}
\end{center}

From now on, we replace the parameter $\lambda$ by $t=\lambda^{-1}$ so that the special fibre is given by $t=0$.

\subsection{The multigraph construction}\label{ss:multigraph}
Let us now take the set-up from the introduction \S\ref{ss:setup}, so that $M$ is a nonsingular algebraic variety and $f:(M,D)\to(\C,0)$ is a regular function such that $D\coloneq f^{-1}(0)$ is a divisor with simple normal crossings, i.e. such that $D=\bigcup_{i\in I}D_i$ with $D_i$ nonsingular hypersurfaces meeting transversally.

For $i\in I$, let $L_i\to M$ be the algebraic line bundle associated to $D_i$. Let $s_i$ be a regular section of $L_i$ such that \eqref{eqn:globalformula} holds.

Following \cite[(3.5)]{GLM06}, we consider the fibre product over $M$ of the graph constructions associated to each $D_i$, i.e.
$$M\times(\C^*)^I \longhookrightarrow \fprod{M}{i\in I}(\PP(L_i\oplus\one)\times\C)\stackrel{\displaystyle \rho}{\longrightarrow} \C^{I}.$$
where the inclusion is given by $(x,(t_i)_{i\in I})\mapsto([s_i(x):t_i],t_i)_{i\in I}$.

Denote by $\PP M$ the Zariski closure of the image of $M\times(\C^*)^{I}$ in $\fprod{M}{i\in I}\left(\PP(L_i \oplus \one)\times\C\right)$. We will give a description of the special fibre $\rho_{|\PP M}^{-1}(0)\subset\PP M$. For $J\subset I$, set $$\pbar{L_J} \coloneq \fprod{\Do_J}{i\in J}\PP\left({L_i}\restrDo\oplus\one\right).$$

\begin{rmk}\label{rmk:singleton}
Denote the projection by $\pr:\rho_{|\PP M}^{-1}(0)\to M$, then
\begin{equation}
\label{eqn:singleton}
\pr^{-1}\left(\Do_J\right)=\left\{\left([s_i(x): 0]\right)_{i\in I\setminus J},\ x\in\Do_J\right\}\times\pbar{L_J}\simeq\left\{\left([1: 0]\right)_{i\in I\setminus J}\right\}\times\pbar{L_J}\simeq\pbar{L_J}
\end{equation}
\end{rmk}

\begin{rmk}
For $J=\varnothing$, note that $\pbar{L_{\varnothing}}=\pr^{-1}(M\setminus D)\simeq M\setminus D$.
\end{rmk}

\begin{rmk}
For $J\subset I$, note that $\PP\left({L_i}_{|{\Do_J}}\oplus\one\right)={L_i}_{|{\Do_J}}\cup\PP\left({L_i}_{|{\Do_J}}\right)$ where $\PP\left({L_i}_{|{\Do_J}}\right)$ is the part of the projective compactification of ${L_i}_{|{\Do_J}}$ at infinity. \\
\end{rmk}

For $J\subset K\subset I$, we set $$\PP_{KJ}\coloneq\fprod{\Do_K}{i\in K}\PP_{KJ,i}\subset\pbar{L_{K}}$$ where \begin{equation}\label{eqn:PKJ}\PP_{KJ,i}\coloneq\left\{\begin{array}{cl}\PP\left({L_i}_{|{\Do_K}}\oplus\one\right)&\text { if }i\in J \\\PP\left({L_i}_{|{\Do_K}}\right)&\text{ if }i\in K\setminus J\end{array}\right.\end{equation}

In particular, $\PP_{KK}=\pbar{L_K}$ and $\PP_{K\varnothing}=\fprod{\Do_K}{i\in K}\PP\left({L_i}_{|\Do_K }\right)$ is the part of $\pbar{L_K}$ at infinity.

The following description is useful for comparison with the real oriented case that we introduce in the next section. The special fibre $\rho_{|\PP M}^{-1}(0)$ may be written as the following disjoint union (up to identification \eqref{eqn:singleton}) $$\rho_{|\PP M}^{-1}(0)=\bigsqcup_{J\subset I}\pbar{L_J}.$$
The closure of $\pbar{L_J}$ in $\rho_{|\PP M}^{-1}(0)$ is the following disjoint union: $$\overline{\pbar{L_J}}=\bigsqcup_{K\supset J}\PP_{KJ}.$$
In particular, the closure of $M\setminus D$ in $\rho_{|\PP M}^{-1}(0)$ is isomorphic to $M$.

\subsection{The motivic Milnor fibre}\label{sec:motivic} In this section, we follow \cite{GLM06} and keep the notation from \S\ref{ss:multigraph}.

\subsubsection{A Grothendieck ring}
\begin{defi}
We first define the category $\VarDCn$ where $n\in\Np$. An object of $\VarDCn$ is a morphism of the form $\varphi:Y\rightarrow D\times\C^*$ where $Y$ is a complex algebraic variety equipped with a good $\C^*$-action, meaning that every orbit is contained in an affine open subset of $Y$ such that
\begin{enumerate}[label=(\roman*)]
\item the fibres of $\pr_D\circ\varphi:Y\rightarrow D$ are $\C^*$-invariant, and
\item $\forall\lambda\in\C^*,\,\forall x\in Y,\,\pi(\lambda\cdot x)=\lambda^n\pi(x)$ where $\pi=\pr_{\C^*}\circ\varphi$.
\end{enumerate}
A morphism of $\VarDCn$ is a $\C^*$-equivariant morphism over $D\times\C^*$:
\begin{center}\begin{tikzcd}
Y_1 \arrow[rd, "\varphi_1"'] \ar[rr] & & Y_2 \arrow[ld, "\varphi_2"] \\
& D\times\C^* &
\end{tikzcd}\end{center}
\end{defi}

\begin{defi}\label{defi:K0}
We denote by $\KDCn$ the Grothendieck group associated to $\VarDCn$; it is the free abelian group generated by the isomorphism classes $\left[\varphi:Y\rightarrow D\times\C^*\right]$ of $\VarDCn$ modulo the following two relations:
\begin{enumerate}[label=(\roman*),ref=\ref{defi:K0}.(\roman*)]
\item If $Z$ is a closed $\C^*$-invariant subvariety of $Y$, then $$\left[\varphi:Y\rightarrow D\times\C^*\right]=\left[\varphi_{|Y\setminus Z}:Y\setminus Z\rightarrow D\times\C^*\right]+\left[\varphi_{|Z}:Z\rightarrow D\times\C^*\right].$$
\item\label{item:liftact} Let $\varphi:\C^*\acts_{\tau}Y\rightarrow D\times\C^*$ be in $\VarDCn$ and set $\psi=\varphi\circ\pr_Y:Y\times\C^m\rightarrow D\times\C^*$. \\
Let $\sigma$ and $\sigma'$ be two good actions on $Y\times\C^m$ such that $\pr_Y(\lambda\cdot_\sigma x)=\pr_Y(\lambda\cdot_{\sigma'}x)=\lambda\cdot_\tau\pr_Y(x)$. Then $\psi:\C^*\acts_{\sigma}(Y\times\C^m)\rightarrow D\times\C^*$ and $\psi:\C^*\acts_{\sigma'}(Y\times\C^m)\rightarrow D\times\C^*$ are in $\VarDCn$ and we add the relation $$\left[\psi:\C^*\acts_{\sigma}(Y\times\C^m)\rightarrow D\times\C^*\right]=\left[\psi:\C^*\acts_{\sigma'}(Y\times\C^m)\rightarrow D\times\C^*\right].$$
\end{enumerate}

The fibre product over $D\times\C^*$ with diagonal action induces a ring structure on $\KDCn$ and the fibre product over $\{\mathrm{pt}\}=\Spec(\C)$ induces a $\K(\Var_\C)$-module structure on $\KDCn$.

Note that the unit of the addition is given by $0\coloneq[\varnothing]$ and that the unit of the product is given by $\II_n\coloneq\left[\id:D\times\C^*\rightarrow D\times\C^*\right]$ where $\lambda\cdot(x,r)=(x,\lambda^nr)$.
\end{defi}

\begin{defi}
We denote by $\prec$ the directed partial order on $\Np$ defined by $m_1\prec m_2\Leftrightarrow\exists k\in\Np,m_1=km_2$. For $m_1\prec m_2$, we set $\theta_{m_2m_1}:\Var_{D\times\C^*}^{\C^*,m_2}\rightarrow\Var_{D\times\C^*}^{\C^*,m_1}$ which only replaces the action by $\lambda\cdot_{m_1} x=\lambda^k\cdot_{m_2}x$. Then we set $\VarDC\coloneq\varinjlim\VarDCn$.

We define similarly $\KDC\coloneq\varinjlim\KDCn$ which is a $\K\left(\Var_\C\right)$-module and set $\II\coloneq\varinjlim\II_n$.
\end{defi}

\begin{rmk}\label{rmk:injlim}
The above inductive limits simply mean that we identify $\varphi:\C^*\acts_{\sigma}Y\rightarrow D\times\C^*$ in $\Var_{D\times\C^*}^{\C^*,n}$ with $\varphi:\C^*\acts_{\sigma'}Y\rightarrow D\times\C^*$ in $\Var_{D\times\C^*}^{\C^*,kn}$ when $\lambda\cdot_{\sigma'}x=\lambda^k\cdot_{\sigma}x$.
\end{rmk}

\subsubsection{The motivic Milnor fibre}\label{ss:mMf}
For $J\subset I$, we set $$L_J^\star \coloneq \fprod{\Do_J}{i\in J}{L^*_i}_{|{\Do_J}} \subset \pbar{L_J}\subset\rho_{|\PP M}^{-1}(0)$$ according to the notation from \S\ref{ss:notation}.

Let us compute $f$ near $L_J^\star$ in the coordinates 
$([s_i(x):t_i],t_i)_{i\in I}\in(\PP(L_i\oplus\one)\times\Ro)$ on $\fprod{M}{i\in I}(\PP(L_i\oplus\one)\times\C) $:
\begin{align*}
f(x) & = u(x) \prod_{i\in I} s_i(x) ^{N_i} =  \prod_{i\in J} t_i ^{N_i} \cdot \left(u(x) \prod_{i\notin J} s_i(x)^{N_i}\right) \cdot  \prod_{i\in J} \left(s_i (x) /t_i\right)^{N_i};
\end{align*}
then 
$$f(x)\cdot\left(\prod_{i\in J}t_i^{N_i}\right)^{-1}\longrightarrow f_J\left(x,(v_i)_{i\in J}\right),$$
as $x\to\Do_J$, $t\to 0$, and $s_i(x)/t_i \to v_i\in L_i ^*$ for $i\in J$, and 
\begin{align}\label{def:f_J}
f_J\left(x,(v_i)_{i\in J}\right) = u_J(x)\prod_{i\in J}v_i^{N_i}
\end{align}
where $\displaystyle u_J(x)=u(x)\prod_{i\notin J}s_i(x)^{N_i}$ is a regular nowhere vanishing section of $\displaystyle\bigotimes_{i=1}^dL_i^{-N_i}$ over $\Do_J$. \\

Given $k\in\Np^J$, we equip $L_J^\star$ with the $\C^*$-action defined by $\lambda\cdot\left(x,(v_i)_{i\in J}\right)=\left(x,\left(\lambda^{k_i} v_i\right)_{i\in J}\right)$.

\begin{prop}\label{prop:actions}
The class $\left[(\pi_J,f_J):\C^*\acts L_J^\star\to D\times\C^*\right]$ is well-defined in $\KDC$ and does not depend on the choice of $k\in\Np^J$.
\end{prop}
\begin{proof}
By additivity we may work with coordinates, up to working over an open Zariski set. \\
Set $N_J=\gcd\{N_i,\,i\in J\}$; then, by Bézout's identity, there exist $\alpha_i\in\Z$ such that $N_J=\sum\alpha_iN_i$. 
Define $\widehat{L_J^\star}\coloneq\left\{(x,r)\in\Do_J\times\C^*,\,u_J(x)r^{N_J}\neq0\right\}$. Then $$\xymatrix{L_J^\star \ar[rr]_-\simeq^-{\Psi_J} \ar[rd]_{(\pi_J,f_J)}&&\widehat{L_J^\star}\times\left\{w\in(\C^*)^J,\,\displaystyle\prod_{i\in J}w_i^{N_i/N_J}=1\right\} \ar[ld]^{\left(x,u_J(x)r^{N_J}\right)} \\ &\Do_J\times\C^*\subset D\times\C^*& },$$ where $$\Psi_J\left(x,(v_i)_{i\in J}\right)=\left(x,r=\prod_{i\in J}v_i^{N_i/N_J},\left(r^{-\alpha_i}v_i\right)_{i\in J}\right),$$ and $$\Psi_J^{-1}(x,r,w)=\left(x,\left(r^{\alpha_i}w_i\right)_{i\in J}\right).$$

For $k\in\Np^J$, the action on $L_J^\star$ induces via $\Psi_J$ the action $\lambda\cdot(x,r)=(x,\lambda^{n/N_J}r)$ on $\widehat{L_J^\star}$ where $n=\sum_{i\in J}k_iN_i$.

Since $N_i>0$, it follows from \ref{item:liftact} that the class induced by $\Psi_J(L_J^\star)$ in $\KDCn$ depends only on the action on $\widehat{L_J^\star}$. Finally, using Remark \ref{rmk:injlim}, the class induced in $\KDC$ does not depend on $n/N_J$, and is, thus, independent of $k$.

\end{proof}

\begin{defi}\label{def:Milnorfibre}
The motivic Milnor fibre of $f$ is defined by
\begin{equation}\label{eqn:Sf}\Sf \coloneq -\sum_{\varnothing\neq J\subset I}(-1)^{|J|}\left[(\pi_J,f_J):\C^*\acts L_J^\star\to D\times\C^*\right]\in\KDC.\end{equation}
\end{defi}

\begin{rmk}\label{rmk:derniere}
The motivic Milnor fibre is usually defined as the formal limit of a power series called the motivic zeta function; see Section \ref{sec:localcase}. This formal limit is defined using a rational formula for the motivic zeta function, which is obtained using resolution of singularities together with the motivic change of variables formula. Nonetheless, the motivic Milnor fibre doesn't depend on the choice of resolution since the motivic zeta function is intrinsically defined in terms of arc spaces.

The rationality formula relies on motivic integration, for which it is necessary to be able to use the multiplicative inverse of the Lefschetz motive $\LL\coloneq [\C]\cdot\II$. Therefore, the motivic Milnor fibre is usually defined as an element of the localised Grothendieck ring $\MDC\coloneq\KDC\left[\LL^{-1}\right]$.

In the present set-up, $D\coloneq f^{-1}(0)$ is already assumed to be a divisor with simple normal crossings, so we can directly use formula \eqref{eqn:Sf} as a definition for the motivic Milnor fibre of $f$ as an element of $\KDC$ (and not $\MDC$); compare with \cite[(3.6.3)]{GLM06}.

The independence of $k$ in Proposition \ref{prop:actions}, and thus relation \ref{item:liftact}, is usually used to prove that the motivic zeta function is rational; however, we use it in Section \ref{sec:localcase} to prove that the (local) motivic Milnor fibre of $f:(X,x_0)\to(\mathbb C,0)$ is actually well-defined without inverting the Lefschetz motive and without the assumption that $f$ is normal crossing (see Theorem \ref{thm:invSf}). This fact has already been noticed by J. Nicaise and S. Payne \cite{NP} and A. Forey \cite[Proposition 5.14]{For19} using other methods.
\end{rmk}

\section{The real oriented graph construction and the complete Milnor fibration}\label{sec:cpl}
Real oriented blowings-up are used to study the fibres of complex algebraic or analytic functions and mappings; see \cite{AC75}, \cite{Par98}, \cite{Kaw02}. The underlying idea is the following. Consider an analytic function (or more generally a mapping as in \cite{Kaw02}) on a nonsingular space. Then, by resolution of singularites, after blowing-up centres included in the special fibre, we may assume that this fibre is a normal crossing divisor. Nonetheless, such an operation on the special fibre does not make the function locally topologically trivial and it is impossible to do it simply by means of complex analytic geometry. On the other hand, the fibre product of real oriented blowings-up of the components of the divisor with normal crossings has this property. Though the space obtained this way is not a complex analytic variety, it carries a lot of information, including the weight filtration and the Hodge filtration; cf. \cite{Kaw02}.

The major contribution of this section is the \emph{real oriented multigraph construction} introduced in \S\ref{ss:rom}. The main idea consists in replacing the blowings-up involved in the multigraph construction, see \S\ref{ss:multigraph}, by real oriented blowings-up. 
The resulting space may be visualized as the ambient space with the divisor with normal crossings replaced by an infinitesimal punctured neighbourhood. Then, in \S\ref{ss:loggeo}, we explain the relation between the spaces obtained with this geometric construction and the log-theoretic spaces from Section \ref{sec:logtheoric}.

\subsection{Examples}\label{ss:examples}
Let us explain the construction of $\Mcpl$, $\Mmot$, and $\Mtop$ on simple examples.

\begin{ex}
For $M=\C$ and $D=\{0\}$ we have $(\mathbb C,0)_0^\cpl=(\mathbb C,0)_0^\mot\sqcup(\mathbb C,0)_0^\top$, where $(\mathbb C,0)_0^\top=S^1$ is the exceptional set of the real oriented blowing-up of $\mathbb C$ at $0$ and $(\mathbb C,0)_0^\mot=\mathbb C^*=\Rp\times S^1$.  
\end{ex}

In general, when $D$ has several components, the inclusion 
$$\Mcpl\supset \Mmot \sqcup \Mtop$$
 may be strict and we call the difference $\Mcpl\setminus(\Mmot \sqcup \Mtop)$ the \emph{mixed part}.  It corresponds to the elements that are "at infinity" only for some, but not all, components of $D$. See the following example (and compare with Example \ref{ex:log}). 

 \begin{ex}\label{ex:ex0}
For $f(x,y)=xy$ defined on $\mathbb C^2$, we get the following schematic picture of $\Mcpl\restrD$.

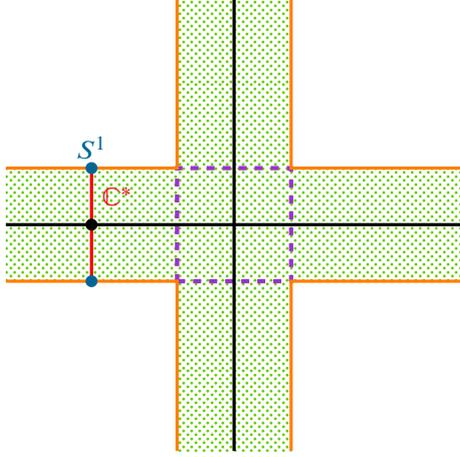
\begin{center}
\begin{minipage}{6cm}
\begin{tikzpicture}[scale=.75]
\fill[pattern=crosshatch dots,pattern color=Chartreuse3,thick] (-4,1) rectangle (-1,-1);
\fill[pattern=crosshatch dots,pattern color=Chartreuse3,thick] (-1,4) rectangle (1,1);
\fill[pattern=crosshatch dots,pattern color=Chartreuse3,thick] (-1,-1) rectangle (1,-4);
\fill[pattern=crosshatch dots,pattern color=Chartreuse3,thick] (1,1) rectangle (4,-1);
\fill[pattern=crosshatch dots,pattern color=Chartreuse3] (-1,-1) rectangle (1,1);
\draw[ultra thick,DarkOrchid3,dashed] (-1,-1) rectangle (1,1);
\draw[very thick,DarkOrange1] (-4,1) -- (-1,1) -- (-1,4);
\draw[very thick,DarkOrange1] (-4,-1) -- (-1,-1) -- (-1,-4);
\draw[very thick,DarkOrange1] (4,1) -- (1,1) -- (1,4);
\draw[very thick,DarkOrange1] (4,-1) -- (1,-1) -- (1,-4);
\draw[very thick] (-4,0) -- (4,0);
\draw[very thick] (0,-4) -- (0,4);
\draw[very thick,red] (-2.5,1) -- (-2.5,-1) node[pos=.25,right] {$\mathbb C^*$};
\fill (-2.5,0) circle (3pt);
\fill[MidnightBlue] (-2.5,1) circle (3pt) node[above] {$S^1$};
\fill[MidnightBlue] (-2.5,-1) circle (3pt);
\end{tikzpicture}
\end{minipage}
\begin{minipage}{7.75cm}
\footnotesize
$\Mmot\restrD$, in green, represents geometrically an infinitesimal tubular neighbourhood of $D$ (with $D$ removed). \\
The boundary of $\Mtop\restrD$, in orange, coincides with the total space of A'Campo's first model of the topological Milnor fibration (see Section \ref{sec:acampo}). \\
The zero section (in black) is not a part of $\Mcpl\restrD$. \\
The square bounded by the purple dashed lines represents $\Mcpl_0$, that is the part over the origin, and the 
purple square itself (without its vertices) represents the mixed  part.  \\
Thus, above the origin, we have: \\
$\Mcpl_0=\Mtop_0\sqcup\Mmot_0\sqcup(M,D)_0^{\mathrm{mix}}$ \\
$\hphantom{\Mcpl_0}=(S^1\times S^1)\sqcup(\mathbb C^*\times\mathbb C^*)\sqcup\big((\mathbb C^*\times S^1)\sqcup(S^1\times\mathbb C^*)\big)$.\\
And above $x\in D \setminus\{0\}$, we have: \\
$\Mcpl_x=S^1\sqcup\mathbb C^*=\Mtop_x\sqcup\Mmot_x$. \\
The circle $S^1$ at infinity of the fibre (in red) is represented by two points in blue: we see only $S^0$ since the picture is real. \\
\end{minipage}
\end{center}
\end{ex}

\subsection{The real oriented graph construction}\label{ss:rog}
Recall that the real oriented blowing-up of $\C$ at the origin is defined by $$\pi=\pr_{\mathbb C}:\widehat{\mathbb C}=\left\{(z,e^{i\theta})\in\mathbb C\times S^1\,:\,|z|e^{i\theta}=z\right\}\to\mathbb C$$ or equivalently, in terms of polar coordinates, by $$\pi:\widehat\C=\Ro\times S^1\to\C,\quad\pi(r,e^{i\theta})=re^{i\theta}.$$
Thus, the exceptional divisor $\pi^{-1}(0)\simeq S^1$ is the space of rays from the origin.

The quotient $\left((\C\times\Ro)\setminus\{(0,0)\}\right)/\Rp$ by the diagonal action $t\cdot (z,r)\coloneq(tz,tr)$ can be identified with the compactification of $\C$ by $S^1$ at infinity (which we denote by $\widehat\PP^1$), or, equivalently, with the real oriented blowing-up of $\PP^1$ at $\infty\coloneq[1:0]$. If we denote the latter by $\pi:\widehat\PP^1\to\PP^1$, then $\pi^{-1}(\infty)\simeq S^1$.

From now on and throughout \S\ref{ss:rog}, we assume that $D$ has a single irreducible component. More precisely, let $M$ be a nonsingular algebraic variety and let $s:M\to L$ be a regular section of an algebraic line bundle $L$ defining a nonsingular hypersurface $D=s^{-1}(0)$.

We define $\hat L\subset L$ by $\hat L_x\coloneq\Rp s(x)$ if $s(x)\neq0$ and $\hat L_x\coloneq L_x\setminus\{0\}$ otherwise. Then the real oriented blowing-up of $M$ along $D$ is $\hat M\coloneq\hat L/\Rp$ the quotient by the action of $\Rp$ on $\hat L$. Note that $\hat M$ depends only on $D$ since if $s'$ is another suitable section then $s'=us$ for some non-vanishing $u$.

We denote by $\Sph(L\oplus\one_{\ge0})$ the fibrewise compactification of $L$ by $S^1$ at infinity. Formally, it is defined on the fibre above $x$ by taking the quotient $\left(\left(L_x\times\Ro\right)\setminus\{(0,0)\}\right)/\Rp$ where the action is diagonal.  Then $\Sph(L\oplus\one_{\ge0})=L\sqcup\Sph(L)$. Here $\one_{\ge0}$ stands for the nonnegative vectors of the trivial $\R$-bundle in $M$.

Consider $$M\times\Rp\longhookrightarrow\Sph(L\oplus\one_{\ge0})\times\Ro\stackrel{\displaystyle\rho}{\longrightarrow}\Ro,$$
 where the inclusion is given by $(x,t)\to\left((s(x),t)\mod\Rp,t\right)$ and $\rho$ denotes the projection onto the second factor. Denote by $\Sph\M$ the Euclidean closure of the image of $M\times\Rp$ in $\Sph(L\oplus\one_{\ge0})\times\Ro$ and by $\rho_{\Sph}$ the restriction of $\rho$ to $\Sph\M$. Then, the \emph{real oriented graph construction} is given by
$$\rho_{\Sph}^{-1}(0)=\Sph(L_{|D}\oplus\one_{\ge0})\bigcup_{\Sph(L_{|D})}\widehat{M},$$
where $\widehat{M}$ is the real oriented blowing-up of $M$ along $D$ and ${\Sph(L_{|D})}$ is its exceptional divisor.

 In particular, we have a copy of $M\setminus D$ at infinity (with respect to the compactification of the vector bundles; compare with \S\ref{ss:McP}) in $\rho_{\Sph}^{-1}(0)$.

We set
\begin{equation}\label{eqn:sphdecomposition}
\Mcpl\coloneq\Sph(L_{|D}^*\oplus\one_{\ge0})\bigcup_{\Sph(L_{|D})}\widehat{M},
\end{equation}
that is, $\Mcpl$ is $\rho_{\Sph}^{-1}(0)$ with the zero section of $L_{|D}$ removed and
\begin{equation}
\Mmot\coloneq \Mcpl \setminus \Sph(L_{|D}) \label{eqn:MmotDi},
\end{equation}
which consists of $L_{|D}^*$ union the image of $M\setminus D$ in $\widehat M$.

Finally we define  $\Mtop$ as the real oriented blowing-up of $M$ along $D$
\begin{equation}
\Mtop\coloneq \widehat{M}\label{eqn:MtopDi},
\end{equation}
which is the part of $\Mcpl$ at infinity and which coincides with the Euclidean closure of $M\setminus D$ in $\Mcpl$.

 Note that the copy of $M\setminus D$ at infinity lies in all three of the above spaces, and is formally given by $M\setminus D \ni x \mapsto (s(x) \mod\Rp) \in \Sph(L)\subset\Mcpl$. Its complement, $\Mcpl\restrD$,  is the disjoint union of $\Mtop\restrD=\Sph(L_{|D})$ and $\Mmot\restrD=L_{|D}^*$.

\begin{rmk}\label{rmk:SP}
The formula \eqref{eqn:sphdecomposition} is analogous to \eqref{eqn:decomposition}. To make this analogy precise, we note that $\PP M$ is the image $\Sph\M$ in $\PP(L\oplus\one)\times\C$ via the projection
$$
\Sph(L\oplus\oneplus)\times\Ro
{\longrightarrow}\PP(L\oplus\one)\times\C .
$$
Indeed, let $([v:1],0)\in\PP M\cap\rho^{-1}(0)$. Then, for some sequences $x\to x_0\in D$ and $t\to0$ with $t\in\C$, $\frac{s(x)}t\to v$. By choosing the local coordinates such that $s(x)=x_1$, we may suppose that $s\left(\frac{|t|}{t} x\right)= \frac{|t|}{t}s(x)$ and then $$\left(\left[s\left(\frac{|t|}{t} x\right):|t|\right], |t| \right) = \left( [s(x):t], |t| \right) \to ([v:1],0).$$
Therefore  $([v:1],0)$  is the image of the projection of the limit of $\left( \left(s\left(\frac{|t|}{t} x\right),|t|\right) \mod\Rp, |t| \right) \in \Sph(L\oplus\oneplus)\times\Ro$.

The projection $\Sph(L_{|D}\oplus\oneplus)\to\PP(L_{|D}\oplus\one)$ is bijective except over the section at infinity. That is to say, $$\Sph(L_{|D}\oplus\oneplus)\setminus\Sph(L_{|D})\simeq\PP(L_{|D}\oplus\one)\setminus\PP(L_{|D})\simeq L_{|D}.$$
\end{rmk}

\subsection{The real oriented multigraph construction}\label{ss:rom}
In the general case -- i.e. when $D=\bigcup_{i\in I} D_i$ has several components -- we consider the fibre product over $M$ of the real oriented graph constructions associated to each $D_i$, i.e. $$M\times\Rp^{I}\longhookrightarrow\fprod{M}{i\in I}\left(\Sph(L_i\oplus\one_{\ge0})\times\Ro\right)\stackrel{\displaystyle\rho}{\longrightarrow}\Ro^I.$$
We denote by $\Sph\M$ the closure of the image of $M\times\Rp^{I}$ in $\fprod{M}{i\in I}\left(\Sph(L_i\oplus\one_{\ge0})\times\Ro\right)$ and the restriction of $\rho$ to $\Sph\M$ by $\rho_{\Sph}$. It coincides with the fibre product of the real oriented graph constructions $\rho_{\Sph,i}: \Sph_i\M \to \R$ associated to the fibre bundles $L_i$.
\begin{defi}\label{defi:Mtop}
We set 
\begin{equation}
\Mcpl\coloneq \fprod{M}{i\in I}\Mcpli\subset\rho_{\Sph}^{-1}(0), \label{eqn:Mext}
\end{equation}
\begin{equation}
\Mmot\coloneq\fprod{M}{i\in I} \Mmoti\subset \Mcpl, \label{eqn:Mmot}
\end{equation}
and,
\begin{equation}
\Mtop\coloneq\fprod{M}{i\in I} \Mtopi\subset\Mcpl, \label{eqn:Mlog}
\end{equation}
where $\Mcpli$, $\Mmoti$, and $\Mtopi$ are defined in \eqref{eqn:sphdecomposition}, \eqref{eqn:MmotDi}, and \eqref{eqn:MtopDi}, respectively.
\end{defi}

\begin{rmk}
The space $\Mcpl$ coincides with $\rho_{\Sph}^{-1}(0)$ minus the zero sections, similar to \eqref{eqn:sphdecomposition}. However, $\Mcpl\restrD$ is no longer the union of $\Mmot\restrD$ and $\Mtop\restrD$ as soon as $|I|\geq 2$ due to the presence of mixed part; see Example \ref{ex:ex0}.
\end{rmk}

Note that, again, the copy of $M\setminus D$ at infinity, formally given by $$M\setminus D \ni x \mapsto (s_i(x) \mod\Rp)_{i\in I} \in \prod_{i\in I} \Sph(L_i)\subset\Mcpl,$$
is included in all above three spaces.

The space $\Mcpl$ may be visualized as $M$ with $D$ replaced by an infinitesimal punctured neighbourhood. Then $\Mmot$ is the infinitesimal part of $\Mcpl$ together with a copy of $M\setminus D$ at infinity, and $\Mtop$ is the part of $\Mcpl$ at infinity.

\begin{rmk}
We can adapt Remark \ref{rmk:SP} to the current setting with several components in $D$ since the components are normal crossings only.
\end{rmk}

Consider the usual stratification of $M$ associated to a divisor $D$ -- i.e. given by $\Do_J\coloneq D_J\setminus\bigcup_{i\in I\setminus J}D_i$ for $J\subset I$, where $D_J\coloneq\bigcap_{i\in J}D_i$. Denote by $\pr:\rho_{\Sph}^{-1}(0)\to M$ the projection to $M$. Then $\pr^{-1}(\Do_J)$ stratifies the special fibre $\rho_{\Sph}^{-1}(0)$ as follows.

For $J=\varnothing$, we get $\pr(\Do_J)=\pr^{-1}(M\setminus D)\simeq M\setminus D$. Otherwise, if $J\neq\varnothing$, we have $$\pr^{-1}(\Do_J)=\left\{\left([s_i(x):0]\right)_{i\in I\setminus J},\,x\in\Do_J\right\}\times \fprod{\Do_J}{i\in J}\Sph({L_i}_{|\Do_J}\oplus\oneplus)\simeq\fprod{\Do_J}{i\in J}\Sph({L_i}_{|\Do_J}\oplus\oneplus).$$
In what follows, for $J\subset I$ we set $$\sbar{L_J}\coloneq\fprod{\Do_J}{i\in J}\Sph({L_i}_{|\Do_J}\oplus\oneplus),$$
and for $J\subset K\subset I$, we set $$\Sph_{KJ}\coloneq\fprod{\Do_K}{i\in K}\Sph_{KJ,i}\subset\sbar{L_{K}},$$
where $$\Sph_{KJ,i}\coloneq\left\{\begin{array}{cl}\Sph\left({L_i}_{|{\Do_K}}\oplus\oneplus\right)&\text{ if }i\in J \\ \Sph\left({L_i}_{|{\Do_K}}\right)&\text{ if }i\in K\setminus J\end{array}\right.$$
Thus, in particular, $\Sph_{KK}=\sbar{L_{K}}$ and $\Sph_{K\varnothing}=\fprod{\Do_K}{i\in K}\Sph\left({L_i}_{|\Do_K }\right)$ is the part of $\sbar{L_{K}}$ at infinity. \\

The special fibre $\rho_{\Sph}^{-1}(0)$ may be written as the following disjoint union of real algebraic varieties of dimension $2\dim_\C M$ (if non-empty)
$$\rho_{\Sph}^{-1}(0) =\bigsqcup_{J\subset I}\sbar{L_J} .$$
The closure of $\sbar{L_J}$ in $\rho_{\Sph}^{-1}(0)$ is the following disjoint union: $$\overline{\sbar{L_J}}=\bigsqcup_{K\supset J}\Sph_{KJ}.$$

The next lemma, Lemma \ref{lem:LJ}, relates 
$$\Mmot_{|D}=\bigsqcup_{\varnothing\neq J\subset I}\Mmot_{|\Do_J}\simeq\bigsqcup_{\varnothing\neq J\subset I}L_J^\star $$
and the motivic Milnor fibre $\Sf$ introduced in \S\ref{ss:mMf}, see  Definition \ref{def:Milnorfibre}.

\begin{lemma}\label{lem:LJ}
For $J\subset I$, $\Mmot_{|\Do_J}\simeq L_J^\star$.
\end{lemma}

\begin{proof}
The isomorphism is given by
$$\begin{array}{ccc}L_J^\star&\to&\Mmot_{|\Do_J}\\(x,(v_j)_{j\in J})&\mapsto&\left(x,(v_j)_{j\in J},(s_i(x) \mod\Rp)_{i\in I\setminus J}\right).\end{array}$$
\end{proof}

Though $L_J^\star$ does not depend on the choice of the sections $s_i$, the space $\Mmot_{|\Do_J}$ does. This is because of the last component $(s_i(x) \mod\Rp)_{i\in I\setminus J}$ in the above formula. This dependence is not essential and the simplest way to get rid of it is to find the model of $\Mmot_{|\Do_J}$ independent of such a choice.  This is, of course $\Mmlog_{|\Do_J}$; see \S\ref{ss:loggeo}.

We denote by $\sbar{L_J^\star}$ the bundle $\sbar{L_J}$ with the zero sections removed. Then $\sbar{L_J^\star}\cap L_J=L_J^\star$. Therefore, for $J\subset I$, we have $\Mcpl\restrDo\simeq\sbar{L_J^\star}$ and $$\Mcpl=\bigsqcup_{J\subset I}\Mcpl\restrDo\simeq\bigsqcup_{J\subset I}\sbar{L_J^\star}.$$

Note that $\Mtop\restrDo$ is the part of $\Mcpl\restrDo$ at infinity; indeed $$\Mtop\restrDo=\left\{\left(s_i(x)\mod\Rp\right)_{i\in I\setminus J},\,x\in\Do_J\right\}\times\fprod{\Do_J}{i\in J}\Sph({L_i}_{|\Do_J})\simeq\fprod{\Do_J}{i\in J}\Sph^\star({L_i}_{|\Do_J}).$$ Thus, \begin{equation}\label{eqn:Mtopstrat}\Mtop=\bigsqcup_{J\subset I}\Mtop\restrDo\simeq\bigsqcup_{J\subset I}\Sph(L_J),\end{equation} where $\Sph^\star(L_J)\coloneq\fprod{\Do_J}{i\in J}\Sph({L_i}_{|\Do_J})$. Again, $\Mtop\restrDo$ depends on the choices of the sections $s_i$.

\subsection{Comparison to the log-spaces}\label{ss:loggeo}

In this section we show that $\Mtop$, $\Mcpl$ and $\Mmot$ can be identified with the Kato--Nakayama log-space $\Mlog$, the complete log-space $\Mclog$, and the algebraic log-space $\Mmlog$, respectively.

\begin{thm}\label{thm:logspaces}
The sections $(s_i)_{i\in I}$ induce bijections
$$
\Mmot \simeq \Mmlog, \quad \Mtop \simeq \Mlog  , \quad  \Mcpl \simeq \Mclog.
$$
\end{thm}

\begin{rmk}
The spaces $\Mcpl$, $\Mtop$ and $\Mmot$, as defined in Definition \ref{defi:Mtop}, depend on the choice of the sections $s_i, i\in I$, whereas this is not the case for the corresponding log-spaces. Therefore the above bijections allow us to get rid of this dependence.

Moreover, these natural bijections can be used to induce geometric and algebraic structures on the log-spaces from the ones on $\Mcpl$, $\Mmot$, $\Mtop$ (see Remark \ref{rmk:top-on-log}).
\end{rmk}

\begin{proof}[Proof of Theorem \ref{thm:logspaces}]
We first construct a bijection between $\Mmlogi\restrDi$ and ${L_i^*}_{|D_i}$. For this, we construct a vector bundle isomorphism between $N_i$ and ${L_i}_{|D_i}$ (which depends on the choice of $s_i$); see \S\ref{ss:normalls} for the notation.

Fix $x\in D_i$ and consider the commutative diagram
\begin{equation}\label{eqn:comm_diagr_s}
\begin{gathered}
\begin{tikzcd}
0 \arrow[r] & T_xD_i \arrow[d] \arrow[r] & T_xM \arrow[d, "T_xs_i"] \arrow[r] & (N_i)_x \arrow[d, "(s_i')_x"] \arrow[r]  & 0 \\
0 \arrow[r] & T_xM \arrow[r, "\text{zero}"] & T_{s_i(x)}L_i \arrow[r] & T_0(L_i)_x=(L_i)_x \arrow[r] & 0
\end{tikzcd}
\end{gathered}
\end{equation}
The middle vertical arrow represents the differential of $s_i:M\to L_i$ at $x$. It induces the vertical isomorphism $(s_i')_x$ on the right-hand side. Then $(s_i')_x(v)=D_xs_i(v)=\partial_vs_i(x)$ is the directional derivative of $s_i$ along $v$ at $x$.
Composing $s_i':N_i^*\to{L_i^*}_{|D_i}$ with the bijection of Lemma \ref{lem:log-and-normal},
we get a bijection $$\Mmlogi\restrDi\ni(x,\Phi )\leftrightarrow w_i\in{L_i^*}_{|D_i}.$$

We now express the above bijection directly, without using the normal bundle. For this let us fix a non-vanishing local section $\ell_i$ of $L_i$.  Then the above isomorphism is explicitely given by $w_i=\Phi \left(\frac{s_i}{\ell_i}\right)\ell_i$. Indeed, write $s_i=g_i\ell_i$ and then, since $g(x)=0$ if $x\in D_i$,
\begin{align}\label{eqn:computation-wi}
w_i=\partial_{v_i}s_i(x)=\partial_{v_i}g(x)\cdot\ell_i+g(x)\cdot\partial_{v_i}\ell_i(x)=\Phi(g_i)\cdot\ell_i+0\cdot\partial_{v_i}\ell_i(x).
\end{align}

Given $x\in D$, recall that $J(x)\coloneq\{i\in I\,:\,x\in D_i\}$ and define the bijection $$\begin{array}{cccl}\Mmlog&\to & \ \Mmot \hfill \\(x,\Phi)&\mapsto&\left(x,\left(w_j=\Phi\left(\frac{s_j}{\ell_j}\right)\ell_j(x)\right)_{j\in J(x)},(s_i(x))_{i\in I\setminus J(x)}\right)\end{array}.$$
Note that the above formula depends on the sections $s_i, i\in J(x)$, but not on the sections $\ell_i$.

By taking the circle bundles we obtain the corresponding bijection for the log-space
$$\begin{array}{cccl}\Mlog &\to & \ \Mtop  \hfill \\(x,\varphi)&\mapsto&\left(x,\left(\theta_j=\varphi\left(\frac{s_j}{\ell_j}\right)\theta(\ell_j(x))\right)_{j\in J(x)},\theta (s_i(x))_{i\in I\setminus J(x)}\right) \end{array}$$
where $\theta(v)\coloneq v\mod\Rp$.

Combining both formulae from above, we get the following bijection
$$\begin{array}{cccl}
\Mclog &\to & \ \Mcpl \hfill \\(x,\varphi,\psi)&\mapsto&\left(x,\left(r_j\theta_j=\psi\left(\frac{s_j}{\ell_j}\right)r(\ell_j(x))\varphi\left(\frac{s_j}{\ell_j}\right)\theta(\ell_j(x))\right)_{j\in J(x)},  \left(\theta (s_i(x)), r(s_i(x))\right)_{i\in I\setminus J(x)} \right)\in\sbar{L_{J(x)}^*}.
\end{array}$$
where $r(v)\coloneq v\mod S^1$.
\end{proof}

\begin{rmk}\label{rmk:top-on-log}
There is a natural topology on the log-spaces; see \cite{KN99}. For $\Mlog$, it coincides with the topology induced by the oriented blowing-up, so that the bijection in Theorem \ref{thm:logspaces} is a homeomorphism. In the present paper, we assume that $\Mlog$, $\Mclog$ and $\Mmlog$ are equipped with the topologies induced by the above bijections.
\end{rmk}

Using the functoriality of $\Mmlog$ applied to $f:(M,D)\to(\C,0)$  we obtain the first vertical map of \eqref{eqn:NiceDiagram}
$$f^{\mlog}\restrD:\Mmlog\restrD \to\mathbb C^*,$$
and the identification $\Mmlog\simeq\Mmot$ from Theorem \ref{thm:logspaces} then induces
$$f^{\mot}\restrD: \Mmot\restrD \to (\mathbb C,0)^\mot_{|0}=\mathbb C^*.$$

Composing the bijections in Lemma \ref{lem:LJ} and Theorem \ref{thm:logspaces}, for $\varnothing\neq J\subset I$, we obtain the bijection $\Theta_J:\Mmlog\restrDo\to L_J^\star$ defined by $\Theta_J(x,\Phi)=\left(x,\left(\Phi\left(\frac{s_j}{\ell_j}\right)\ell_j(x)\right)_{j\in J}\right)$.

\begin{prop}\label{prop:reduction}
$f^{\mlog}\restrDo = f_J\circ\Theta_J$, where $f_J$ is defined in \eqref{def:f_J}.
\end{prop}
\begin{proof}
For $(x,\Phi) \in \Mmlog\restrDo$,
\begin{align*}
f^{\mlog}(x,\Phi) & = \Phi(f_x)
=  \Phi\left(  u \prod_{i\in I}
\left (\frac{s_i}{\ell_i } \right)^{N_i}  \prod_{i\in I} \ell_i^{N_i}\right)\\
& = u(x) \prod_{i\notin J} \left( \frac{s_i(x)}{\ell_i (x)} \right)^{N_i}  \Phi\left( 
\prod_{j\in J} \left (\frac{s_j}{\ell_j } \right)^{N_j} \right) \prod_{i\in I} \ell_i(x)^{N_i}\\
& = \left (u(x) \prod_{i\notin J} s_i(x)^{N_i} \right ) \prod_{j\in J} \left (\Phi\left( 
\frac{s_j}{\ell_j } \right) {\ell_j(x) }\right)^{N_j} = f_J\left(\Theta_J(x,\Phi)\right).
\end{align*}
\end{proof}

\begin{defi}\label{def:reduction}
In Theorem \ref{thm:A}, by the reduction of $f^\mlog\restrD$ we mean the element of $\KDC$ defined by $$-\sum_{\varnothing\neq J\subset I}(-1)^{|J|}\left[(\pr_M,f^\mlog\restrD):\C^*\acts\Mmlog\restrDo\to D\times\C^*\right].$$
\end{defi}

We derive from Proposition \ref{prop:reduction} that this reduction coincides with the motivic Milnor fibre $\Sf$; cf. Definition \ref{def:Milnorfibre}.
\begin{cor}\label{cor:reduction}
$\displaystyle \Sf=-\sum_{\varnothing\neq J\subset I}(-1)^{|J|}\left[(\pr_M,f^\mlog\restrD):\C^*\acts\Mmlog\restrDo\to D\times\C^*\right]$.
\end{cor}

Similarly, the functoriality of $\Mlog$ gives rise to a map $\Mlog\restrD\to(\mathbb C,0)^\log_{|0}=S^1$. Then, using the identification $\Mlog\simeq\Mtop$, we obtain
\begin{equation}\label{eqn:TMF}
\sign f:\Mtop_{|D}\to S^1.
\end{equation}
We will see in Proposition \ref{prop:AC} that $\sign f$ coincides with A'Campo's first model of the topological Milnor fibration. \\

We can now prove Theorem \ref{thm:A}, as stated in the introduction.

\begin{proof}[Proof of Theorem \ref{thm:A}]
The fact that the map $f^\log\restrD:\Mlog\restrD\to S^1$ coincides with A'Campo's first model of the topological Milnor fibration will be proved in Section \ref{sec:acampo}, see Proposition \ref{prop:AC}.

We have shown in Corollary \ref{cor:reduction} that the reduction of $f^\mlog\restrD$ coincides with the motivic Milnor fibre $\Sf$.

By definition of the log-spaces, we obtain $f^\log\restrD$ by dividing  $f^\mlog\restrD$ over each stratum $\Do_J$ by $(\Rp)^{J}$ in the source and $\Rp$ in the target.
\end{proof}

\begin{rmk}\label{rmk:speculations}
The signs appearing in the motivic Milnor fibre $\Sf$ (see Definition \ref{def:Milnorfibre}) can be interpreted in terms of geometric quotients by powers of $\Rp$.

Since $\Mtop\restrDo$ is the geometric quotient of $\Mmot\restrDo$ by $\Rp^{|J|}$, the coefficient $(-1)^{|J|}$ appearing in $\Sf$ in front of the class of $L_J^\star$ can be interpreted as a consequence of this fact as $\chi_c(\Rp^{|J|})=(-1)^{|J|}$. We conjecture that this interpretation can be made formal by introducing a Grothendieck group handling these quotients by $\Rp$ (importantly, by taking such a quotient, we leave the category of algebraic varieties).

The additional sign in front of the sum for $\Sf$ in Definition \ref{def:Milnorfibre} is necessary to compare the motivic Milnor fibre and the topological Milnor fibration. Indeed, we need to take the quotient of $\C^*$ by $\Rp$ in order to recover $S^1$ in the target of the topological Milnor fibration.

Besides, as explained in Remark \ref{rmk:speculation2}, these coefficients are also necessary to ensure the invariance with respect to the choice of a resolution (see also Examples \ref{ex:bu} and \ref{ex:bu2}).
\end{rmk}

\subsection{Recovering the function $f$ and the motive $\Sf$}\label{ss:strat}
Recall that $f^\mot:\Mmot_{|D}\to\C^*$ induces maps $f_J:L_J^\star\to\mathbb C^*$ for $\varnothing\neq J\subset I$ (see \eqref{def:f_J}) from which we derive the motivic Milnor fibre $\Sf$ (see Definition \ref{def:Milnorfibre}). Similarly, we deduce from \eqref{eqn:TMF} that $\sign f$ induces an ${(S^1)}^{J}$-equivariant function
\begin{align}\label{eq:signf_J}\sign f_J:\Mtop\restrDo\to S^1, \quad \sign f_J(x,\theta)=(\sign u_J(x))\prod\theta_i^{N_i},\end{align}
where $\varnothing\neq J\subset I$. \\

Actually, we show that we can essentially recover $f$ from these maps, as explained below. For this statement, it is convenient to use the following evaluation map $$\ev_f:\Mmlog\to \C^*$$ given for $\Phi\in\Mmlog_x$ by $\ev_f(\Phi)=\Phi(f_x)$. Note that $(\ev_f)\restrD$ coincides with $f^{\mlog}\restrD:\Mmlog\restrD \to(\mathbb C,0)^\mlog_{|0}=\mathbb C^*$.

\begin{thm}\label{thm:Recover}
Given a manifold $M$ and a divisor $D=\cup D_i$ with simple normal crossings.  Fix the line bundles $L_i$ and the sections $s_i$ as in the set-up, see \S\ref{ss:setup}. 

Then the map $\ev_f:\Mmlog\to\C^*$ uniquely determines the function $f$, and its restriction $f_J^\mlog:\Mmlog\restrDo\to\C^*$ determines $f_J:L_J^\star\to\mathbb C^*$.

Similarly, $f^\log:\Mlog\to S^1$ determines $f$ up to the multiplication by a (locally constant) positive real number and each of the maps $f^\log\restrDo:\Mlog\restrDo\to S^1$ determines $f_J:L_J^\star\to\mathbb C^*$ also up to multiplication by a (locally constant) positive real number.
\end{thm}

\begin{proof}
Clearly $\ev_f$ determines the exponents $N_i$ and therefore it is enough to determine the unit $u$ in \eqref{eqn:globalformula}. If $\Phi\in\Mmlog_x$ this follows from $\Phi(u)=u(x)$.

The above argument applies to $f^\log$ because for $\varphi \in\MD^\log_x$, $f^\log(x,\varphi) = \varphi(f_x)$. More precisely, it allows us to recover the exponents $N_i$ from $f^\log\restrDo$ and then the argument of $u_J$, $\sign u_J(x)=\frac{u_J(x)}{|u_J(x)|}$. Finally, the argument of $u_J$ determines $u_J$ up to a positive real constant. Indeed, the quotient of two such possible $u_J$ is a holomorphic function that takes only real values. So it is locally constant.
\end{proof}

\begin{proof}[Proof of Theorem \ref{thm:B}]
Note $\Xi:\Mlog\to\Mtop$ the bijection from Theorem \ref{thm:logspaces}. Then $f^\log\restrDo=\sign f_J\circ\Xi$. Therefore $\sign f_J$ determines $f^\log\restrDo$, which determines $f_J$ up to a multiplication by a (locally constant) positive real number by Theorem \ref{thm:Recover}.

Note that $\left[(\pi_J,f_J):\C^*\acts L_J^\star\to D\times\C^*\right]\in\KDC$ does not change if we multiply $f_J$ by a (locally constant) positive real number. Therefore $\sign f_J$ determines $\Sf$ (see \eqref{eqn:Sf}).
\end{proof}

\section{The topological Milnor fibration}\label{sec:acampo}
In \cite{AC75}, N. A'Campo gives two geometric models of the topological Milnor fibration. We recall below the main idea of his constructions and interpret them in terms of our construction. 

Let $f:M\to \{z\in \C\,:\,|z|\le 1\}$ be a complex analytic function with values in the unit disc in $\C$, and assume that the zero set of $f$ is a hypersurface with simple normal crossings $D\coloneq f^{-1}(0)=\cup D_i$.
A'Campo constructs a first model $nf:N\to S^1$ (denoted $nP:N\to S^1$ in \cite{AC75}) of the topological Milnor fibration of $f$, with $N$ a differentiable manifold with corners that is homeomorphic to the boundary of a regular neighborhood of the zero set of $f$ in $M$. When $D$ is irreducible, this construction is simply the real oriented blowing-up of $M$ along $D$. For general $D$, he takes the fibre product of the real oriented blowings-up along the components $D_i$.  Therefore, by construction, the fibre of $nf:N\to S^1$ above a smooth point of $f^{-1}(0)$  is a circle, whereas above a point in the intersection of $k$ divisors it is the product of $k$ circles.

Above each stratum $\Do_J$ separately, this construction describes nicely the fibration together with the monodromy. However these models of the monodromy above the different strata do not glue together. For this reason, A'Campo constructs a second model of the Milnor fibration by thickening the corners of $N$, giving more flexibility to interpolate the monodromy on the pieces into a global geometric monodromy.

A'Campo’s first model coincides with the map defined in \eqref{eqn:TMF}. The second one, which we interpret now in terms of our construction, comes with a simple formula for the geometric monodromy operator. \\

Let us now go back to the set-up from the introduction (see \S\ref{ss:setup}). For $\varnothing\neq J\subset I$, recall that $\sign f$ defined in \eqref{eqn:TMF} induces an ${(S^1)}^{J}$-equivariant function (see \eqref{eq:signf_J})
$$\sign f_J:\Mtop\restrDo\to S^1, \quad \sign f_J(x,\theta)=(\sign u_J(x))\prod\theta_i^{N_i}.$$
Recall also that $\Mtop\restrDo$ can be identified with $L_J^\star/(\Rp)^{J}$ so that $\Mtop\restrDo$ is an ${(S^1)}^{J}$-principal bundle over $\Do_J$.

\begin{prop}\label{prop:AC}
The continuous map $\sign f:\Mtop_{|D}\to S^1$ coincides with A'Campo's first model of the topological Milnor fibration.
\end{prop}
\begin{proof}
When $D$ is irreducible, we defined $\Mtop$ as the real oriented blowing-up of $M$ along $D$, see \eqref{eqn:MtopDi}. When $D$ has several components $D=\cup D_i$, we defined $\Mtop$ as the fibre product of the $(M,D_i)^\top$, see \eqref{eqn:Mlog}. Therefore $\Mtop$ coincides with the total space of A'Campo's first model.

Then $\sign f_J:\Mtop\restrDo\to S^1$, defined by $\sign f_J(x,\theta)=(\sign u_J(x))\prod\theta_i^{N_i}$, coincides with A'Campo's first model of the topological Milnor fibration using polar coordinates, see \cite[p239]{AC75}.
\end{proof}

The map $\sign f:\Mtop_{|D}\to S^1$ is a proper stratified submersion and, therefore, a locally trivial fibration. Over each stratum (see \eqref{eq:signf_J}) $\sign f_J$ can be equipped with a monodromy map given by a simple formula.  These monodromy maps do not glue together to a 
continuous monodromy map for $\sign f:\Mtop_{|D}\to S^1$. This is why A'Campo improves his model by adjoining additional simplices at corners, allowing him to define a geometric monodromy for $\sign f:\Mtop_{|D}\to S^1$. It is this construction that we are going to recover using $\sign f:\Mcpl\restrD\to S^1$.

Along the way, we give two interpretations of A'Campo's second model; see Propositions \ref{prop:AC1} and \ref{prop:AC2}. First, we represent A'Campo's model as a subspace of $\Mcpl$. Second, we show that this model can be obtained from $\Mcpl\setminus\Mtop$ by dividing by the diagonal $\Rp$-action. 

For $i\in I$ and $J\subset I$ consider the function
$$\xi_i:L_J^\star\to\R, \quad
\xi_i(x,v) =
\begin{cases}
  (|v_i|N_i)^{-1} & \text{ if } i\in J \\
  0 & \text{ if } i\not \in J
\end{cases}
$$
These functions can be extended continuously from $L_J^\star$ to $\Mcpl\restrDo$ by giving to $\xi_i$ the value $0$ if $v_i$ goes to infinity. 
For $i\in I$, the above functions glue together to a continuous function $\xi_i:\Mcpl\restrD\to\R$. Then we define $\xi:\Mcpl\restrD \to\R^I$ by $\xi=(\xi_i)_{i\in I}$. 

Let $\Delta=\left\{\xi\in\R^I\,:\,\xi\ge0,\,\sum\xi_i=1\right\}$ be the standard simplex in $\R^I$. Then the restriction of $\sign f$ to $\xi^{-1}(\Delta)$ coincides with A'Campo's second model. Indeed, the process of adjoining simplices to the normal crossing divisor $D$ in \cite{AC75} gives exactly the topological space $\xi^{-1}(\Delta) /(S^1)^I$. Besides, we have $$\xi^{-1}(\Delta)=\left(\xi^{-1}(\Delta)/(S^1)^I\right)\times_D\Mtop.$$ In particular, $\xi^{-1}(\Delta)$ is homeomorphic to A'Campo's second model (named $\tilde N$ in \cite{AC75}). 

Then we construct the associated geometric monodromy as follows. For $\lambda\in\R$, we define $h_{\lambda,J}:L_J^\star\to L_J^\star$ by 
$$h_{\lambda,J}\left(x,v \right)=\left(x,\left(  \exp\left(\lambda \xi_i(x,v) 2\pi \sqrt{-1}\right) v_i \right)_{i\in J}\right).$$
Then, the maps $h_{\lambda,J}$ can be extended to the boundary of $L_J^\star$ in $\Mcpl\restrD$ so that they glue to a continuous map $h_\lambda:\Mcpl\restrD\to\Mcpl\restrD$.

The map $h_{\lambda}: \xi^{-1}(\Delta) \to \xi^{-1}(\Delta)$ is well-defined and satisfies
$$\sign f\left(h_{\lambda}(x,v)\right)=\exp\left(\lambda  2\pi \sqrt {-1}\right) \sign f(x,v).$$

In particular, the following result holds:

\begin{prop}\label{prop:AC1}
$\sign f:\xi^{-1}(\Delta)\to S^1$ coincides with A'Campo's second model of the Milnor fibration.
\end{prop}

Let us give a different interpretation of this construction. 
For $\varnothing\neq J\subset K$, we define an $\Rp$-action on
$$\left(\fprod{\Do_K}{i\in J} {L_i^*}_{|\Do_K}\right) \times_{\Do_K} \left(\fprod{\Do_K}{i\in K\setminus J}\Sph\left({L_i}_{|\Do_K}\right)\right) \subset \Sph_{KJ},$$
by $$\lambda\cdot\left((v_i)_{i\in J}, (\theta_j)_{j\in K\setminus J}\right) = \left((\lambda v_i)_{i\in J}, (\theta_j)_{j\in K\setminus J}\right).$$
These actions glue to an $\Rp$-action on $\Mcpl\setminus\Mtop$.  The above construction leads to the following result.

\begin{prop}\label{prop:AC2}
The quotient $\left(\Mcpl\setminus\Mtop\right)/\Rp$ is homeomorphic to $\xi^{-1}(\Delta)$ and the map to $S^1$ induced on this quotient equals the restriction of $\sign f$ to $\xi^{-1}(\Delta)$.
\end{prop}

\begin{proof}[Proof of Theorem \ref{thm:C}]
Theorem \ref{thm:C} follows from Propositions \ref{prop:AC1} and \ref{prop:AC2}.
\end{proof}

\section{Effect of a blowing-up}\label{sec:bl}
We use the same notation as in the introduction, see \S\ref{ss:setup}. Let $C\subset D$ be a nonsingular subvariety in normal crossings with $D=\cup_{i\in I} D_i$.  In this section we compute the effect of the blowing-up of $C$ on $\Mmlog$, $\Mlog$ and $\Mclog$.

Let $\sigma:\tilde M\to M$ denote the blowing-up of $C$ in $M$ and let $\tilde D $ be the total transform of $D$ by $\sigma$.  Thus $\tilde D= \sigma^{-1} (D)$ is the union of the strict transforms of the $D_i$, $i\in I$, denoted by $\tilde D_i$, and of the exceptional divisor $E$ of $\sigma$. We are going to express $\tMmlog$, $\tMlog$, and $\tMclog$ in terms of the original log-spaces and the normal bundles of the components of $D$ and to $C$.  For this we present two types of formulae. 
The first ones, Theorems \ref{thm:details-blow-up-mot} and \ref{thm:details-blow-up-log},   give canonical formulae for the new spaces over strata of the exceptional divisor. The second ones, Corollaries \ref{cor:details-blow-up-mot1} and \ref{cor:details-blow-up-top1}, present (local for the Zariski topology on $C$) trivialisations over $C$ of the projections $\tMmlog\to \Mmlog$, $\tMlog\to \Mlog$.  Over the smallest stratum, denoted later by ${\tildeDoK}$, both coincide and the canonical formula is a trivialisation; see Propositions \ref{prop:details-blow-up-mot2} and \ref{prop:details-blow-up-top2}.  Note, nevertheless, that ${\tildeDoK}$ can be empty.

The case $\tMmlog\to \Mmlog$ is particularly simple and can be expressed entirely in terms of the normal bundles. The case of $\tMlog\to \Mlog$ is, in turn, obtained by dividing, stratumwise, by powers of $\Rp$.  In this case, the local stratumwise trivialisations glue to a continuous (local with respect to $C$) trivialisation over all the strata; see Corollary \ref{cor:glued-blow-up-top1}. We also show that the fibres of the projection $\tMlog\to \Mlog$ are contractible, and, hence, that this projection is a homotopy equivalence; see Corollary \ref{cor:homotopyequiv}. For completeness, we also show how to obtain local (in $C$) trivialisations of  $\tMclog\to \Mclog$ by combining the previous two cases; see Theorems \ref{thm:details-blow-up-cpl1} and \ref{thm:details-blow-up-cpl2}.  The proofs rely on the identifications of Corollary \ref{cor:log-and-normal}.

The underlying ideas are simple and geometric, though the precise statements are technical and the notation that takes into account all the data is quite heavy.  For this reason, we  begin with two simple but illustrative examples that we present with schematic pictures.

\begin{ex}\label{ex:bu}
In this example, we assume that $D:x_1x_2=0$ in $M=\C^2$ and that $C=\{O\}$.
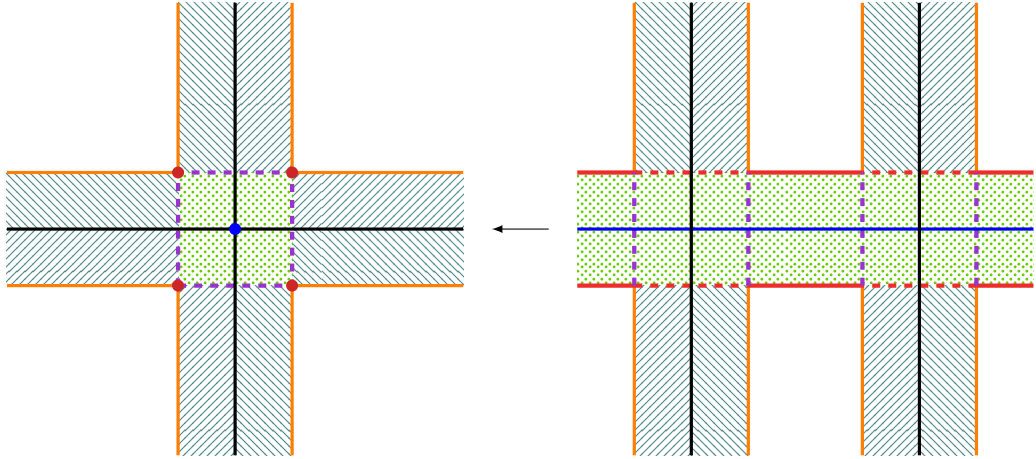
\begin{center}
\begin{tikzpicture}[scale=.75]
\fill[pattern=north west lines, pattern color=DarkSlateGray4,thick] (-4,1) rectangle (-1,0);
\fill[pattern=north east lines, pattern color=DarkSlateGray4,thick] (-4,-1) rectangle (-1,0);
\fill[pattern=north east lines, pattern color=DarkSlateGray4,thick] (4,1) rectangle (1,0);
\fill[pattern=north west lines, pattern color=DarkSlateGray4,thick] (4,-1) rectangle (1,0);
\fill[pattern=north west lines, pattern color=DarkSlateGray4,thick] (-1,4) rectangle (0,1);
\fill[pattern=north east lines, pattern color=DarkSlateGray4,thick] (1,4) rectangle (0,1);
\fill[pattern=north east lines, pattern color=DarkSlateGray4,thick] (-1,-4) rectangle (0,-1);
\fill[pattern=north west lines, pattern color=DarkSlateGray4,thick] (1,-4) rectangle (0,-1);
\fill[pattern=crosshatch dots,pattern color=Chartreuse3] (-1,-1) rectangle (1,1);
\draw[very thick,DarkOrange1] (-4,1) -- (-1,1) -- (-1,4);
\draw[very thick,DarkOrange1] (-4,-1) -- (-1,-1) -- (-1,-4);
\draw[very thick,DarkOrange1] (4,1) -- (1,1) -- (1,4);
\draw[very thick,DarkOrange1] (4,-1) -- (1,-1) -- (1,-4);
\draw[ultra thick,DarkOrchid3,dashed] (-1,-1) rectangle (1,1);
\fill[Firebrick3] (-1,-1) circle (3pt);
\fill[Firebrick3] (-1,1) circle (3pt);
\fill[Firebrick3] (1,-1) circle (3pt);
\fill[Firebrick3] (1,1) circle (3pt);
\draw[very thick] (-4,0) -- (4,0);
\draw[very thick] (0,-4) -- (0,4);
\fill[blue] (0,0) circle (3pt);
\begin{scope}[shift={(10,0)}]
\fill[pattern=crosshatch dots,pattern color=Chartreuse3] (-4,-1) rectangle (4,1);
\fill[pattern=north west lines, pattern color=DarkSlateGray4,thick] (-3,4) rectangle (-2,1);
\fill[pattern=north east lines, pattern color=DarkSlateGray4,thick] (-1,4) rectangle (-2,1);
\fill[pattern=north east lines, pattern color=DarkSlateGray4,thick] (-3,-4) rectangle (-2,-1);
\fill[pattern=north west lines, pattern color=DarkSlateGray4,thick] (-1,-4) rectangle (-2,-1);
\fill[pattern=north west lines, pattern color=DarkSlateGray4,thick] (1,4) rectangle (2,1);
\fill[pattern=north east lines, pattern color=DarkSlateGray4,thick] (3,4) rectangle (2,1);
\fill[pattern=north east lines, pattern color=DarkSlateGray4,thick] (1,-4) rectangle (2,-1);
\fill[pattern=north west lines, pattern color=DarkSlateGray4,thick] (3,-4) rectangle (2,-1);
\draw[very thick,DarkOrange1] (-3,-4) -- (-3,-1);
\draw[very thick,DarkOrange1] (-1,-4) -- (-1,-1);
\draw[very thick,DarkOrange1] (3,-4) -- (3,-1);
\draw[very thick,DarkOrange1] (1,-4) -- (1,-1);
\draw[very thick,DarkOrange1] (-3,1) -- (-3,4);
\draw[very thick,DarkOrange1] (-1,1) -- (-1,4);
\draw[very thick,DarkOrange1] (3,1) -- (3,4);
\draw[very thick,DarkOrange1] (1,1) -- (1,4);
\draw[ultra thick,Firebrick2] (-4,1) -- (-3,1);
\draw[ultra thick,Firebrick2] (-1,1) -- (1,1);
\draw[ultra thick,Firebrick2] (3,1) -- (4,1);
\draw[ultra thick,Firebrick2] (-4,-1) -- (-3,-1);
\draw[ultra thick,Firebrick2] (-1,-1) -- (1,-1);
\draw[ultra thick,Firebrick2] (3,-1) -- (4,-1);
\draw[ultra thick,dashed,Firebrick2] (-3,1) -- (-1,1);
\draw[ultra thick,dashed,Firebrick2] (3,1) -- (1,1);
\draw[ultra thick,dashed,Firebrick2] (-3,-1) -- (-1,-1);
\draw[ultra thick,dashed,Firebrick2] (3,-1) -- (1,-1);
\draw[ultra thick,DarkOrchid3,dashed] (-3,-1) -- (-3,1);
\draw[ultra thick,DarkOrchid3,dashed] (-1,-1) -- (-1,1);
\draw[ultra thick,DarkOrchid3,dashed] (3,-1) -- (3,1);
\draw[ultra thick,DarkOrchid3,dashed] (1,-1) -- (1,1);
\draw[very thick, blue] (-4,0) -- (4,0);
\draw[very thick] (-2,-4) -- (-2,4);
\draw[very thick] (2,-4) -- (2,4);
\end{scope}
\draw[latex-] (4.5,0) -- (5.5,0);
\end{tikzpicture}
\end{center}

The picture on the left represents the space $\Mclog\restrD$. 
The subspace $\Mlog\restrD$ is represented by orange straight lines together with the red dots. The subspace $\Mmlog\restrD$ is represented by hatched green lines outside the intersection and with green dots near the intersection point. 
The mixed part is represented in purple.

The picture on the right represents $\tMD^\clog_{|\tilde D}$. The exceptional divisor $E$ is represented as the horizontal blue line.  The subspace $\tMD^\log_{|\tilde D}$ is represented by straight lines. The subspace $\tMD^\mlog_{|\tilde D}$ is represented by green hatched lines and dots. 
The mixed part is represented by dashed lines.

The color code shows how each piece is mapped by $\sigma$ (for $\tilde D$) or by $\sigma^{\clog}$ (for $\tMclog_{|\tilde D}$). In particular, some pieces of the mixed part of $\tMD^\clog_{|\tilde D}$ are mapped to $\Mlog_{|\{O\}}$.

\noindent Let us start with the $\mlog$-spaces. After blowing-up, $\Mmlog_{|\{O\}}\simeq \C^*\times\C^*$ is replaced by
\begin{itemize}[nosep]
\item $\tMmlog_{|\{p\}}\simeq\C^*\times\C^*$ for $p\in E\cap\tilde D=\{p_1,p_2\}$, and
\item $\tMmlog_{|\{p\}}\simeq\C^*$ for $p\in E\setminus\tilde D\simeq\C^*$.
\end{itemize}

\noindent Concerning the $\log$ spaces, the blowing-up replaces $\Mlog_{|\{O\}}\simeq S^1\times S^1$ with
\begin{itemize}[nosep]
\item $\tMlog_{|\{p\}}\simeq S^1\times S^1$ for $p\in E\cap\tilde D=\{p_1,p_2\}$, and
\item $\tMlog_{|\{p\}}\simeq S^1$ for $p\in E\setminus\tilde D\simeq\C^*\simeq S^1\times \Rp$.
\end{itemize}
Therefore, as sets, above the origin we get $$S^1\times S^1\leftarrow(S^1\times S^1)\sqcup(S^1\times S^1)\sqcup(S^1\times\C^*)=(S^1\times S^1)\times[0,\infty].$$
\end{ex}

\begin{ex}\label{ex:bu2}
We now assume that $D:x_1=0$ in $M=\C^2$ and that $C=\{O\}$.
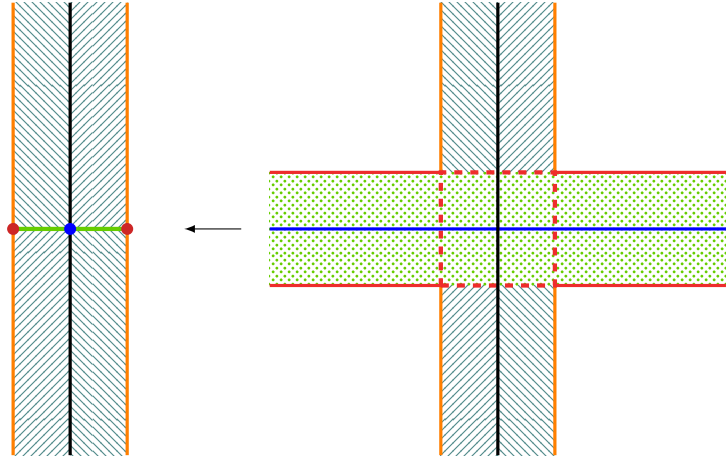
\begin{center}
\begin{tikzpicture}[scale=.75]
\fill[pattern=north west lines, pattern color=DarkSlateGray4,thick] (-1,4) rectangle (0,0);
\fill[pattern=north east lines, pattern color=DarkSlateGray4,thick] (1,4) rectangle (0,0);
\fill[pattern=north east lines, pattern color=DarkSlateGray4,thick] (-1,-4) rectangle (0,0);
\fill[pattern=north west lines, pattern color=DarkSlateGray4,thick] (1,-4) rectangle (0,0);
\draw[ultra thick,Chartreuse3] (-1,0) -- (1,0);
\draw[very thick,DarkOrange1] (-1,4) -- (-1,-4);
\draw[very thick,DarkOrange1] (1,4) -- (1,-4);
\fill[Firebrick3] (-1,0) circle (3pt);
\fill[Firebrick3] (1,0) circle (3pt);
\draw[very thick] (0,-4) -- (0,4);
\fill[blue] (0,0) circle (3pt);
\begin{scope}[shift={(7.5,0)}]
\fill[pattern=north west lines, pattern color=DarkSlateGray4,thick] (-1,4) rectangle (0,1);
\fill[pattern=north east lines, pattern color=DarkSlateGray4,thick] (1,4) rectangle (0,1);
\fill[pattern=north east lines, pattern color=DarkSlateGray4,thick] (-1,-4) rectangle (0,-1);
\fill[pattern=north west lines, pattern color=DarkSlateGray4,thick] (1,-4) rectangle (0,-1);
\fill[pattern=crosshatch dots,pattern color=Chartreuse3] (-4,-1) rectangle (4,1);
\draw[very thick,DarkOrange1] (-1,4) -- (-1,1);
\draw[very thick,DarkOrange1] (1,4) -- (1,1);
\draw[very thick,DarkOrange1] (-1,-4) -- (-1,-1);
\draw[very thick,DarkOrange1] (1,-4) -- (1,-1);
\draw[very thick,Firebrick2] (-4,1) -- (-1,1);
\draw[very thick,Firebrick2] (4,1) -- (1,1);
\draw[very thick,Firebrick2] (-4,-1) -- (-1,-1);
\draw[very thick,Firebrick2] (4,-1) -- (1,-1);
\draw[ultra thick,Firebrick2,dashed] (-1,-1) rectangle (1,1);
\draw[very thick, blue] (-4,0) -- (4,0);
\draw[very thick] (0,-4) -- (0,4);
\end{scope}
\draw[latex-] (2,0) -- (3,0);
\end{tikzpicture}
\end{center}

The straight green line represents the restriction of $\Mmlog$ above $O$ (as a set it coincides with $\C^*$), and the red dots represent the restriction of $\Mlog$ above $O$ (as a set it coincides with $S^1$). 

Note that the blowing-up creates a mixed part in $\tMcpl_{|\tilde D}$ that is mapped to $\Mlog_{|\{O\}}$. 

\noindent Let us start with the $\mlog$ spaces. After blowing-up, $\Mmlog_{|\{O\}}\simeq \C^*$ is replaced by
\begin{itemize}[nosep]
\item $\tMmlog_{|\{p\}}\simeq\C^*\times\C^*$ for $p$ in the singleton $E\cap\tilde D$, and
\item $\tMmlog_{|\{p\}}\simeq\C^*$ for $p\in E\setminus\tilde D\simeq\C$.
\end{itemize}

\noindent Concerning the $\log$ spaces, the blowing-up replaces $\Mlog_{|\{O\}}\simeq S^1$ with
\begin{itemize}[nosep]
\item $\tMlog_{|\{p\}}\simeq S^1\times S^1$ for $p$ in the singleton $E\cap\tilde D$, and
\item $\tMlog_{|\{p\}}\simeq S^1$ for $p\in E\setminus\tilde D\simeq\C$.
\end{itemize}
Therefore, as sets, above the origin we get $$S^1\leftarrow(S^1\times S^1)\sqcup(S^1\times\C)\simeq S^1\times\left(\C\times\R_{\ge0}\setminus\{0,0\}\right)/\Rp,$$
so that the fibre is the join of a point and a circle, that is the cone over a circle. 
\end{ex}

\subsection{Set-up}
Recall that $\sigma:\tilde M\to M$ denotes the blowing-up of $M$ along a nonsingular subvariety 
$C\subset D$ that is in normal crossings with $D=\cup_{i\in I} D_i$. In particular, for every $i\in I$, either $C$ is included in $D_i$ (in that case we say $i\in K\subset I$), $C$ intersects $D_i$ transversally (in that case we say $i\in\Lambda\subset I$), or $C$ does not meet $D_i$. In the latter situation, the blowing-up along $C$ has no effect on $D_i$, so we assume without loss of generality that $I=K\cup\Lambda$. By the assumption $C\subset D$, the index set $K\coloneq\left\{i\in I\,:\,C\subset D_i\right\}$ is non-empty. 

Denote by $\tilde D_i$ the strict transform of $D_i$ and by $\tilde D_0\coloneq E$ the exceptional divisor of $\sigma$. For $J\subset I$ set \begin{equation}\label{eqn:tildeJ}\tilde J\coloneq J\cup \{0\}.\end{equation}
Then 
\begin{align}\label{eqn:tildeD}\tilde D^\circ_{\tilde J}\coloneq \bigcap_{j\in\tilde J}\tilde D_j\setminus\bigcup_{i\in \tilde I\setminus\tilde J}D_i =\left(E\cap\bigcap_{j\in J}\tilde D_j\right) \setminus \bigcup_{i\in I\setminus J} \tilde D_i.
\end{align}
Note that if $C=D_K$ then $\tilde D^\circ_{\tilde K}$ is empty.

\begin{rmk}\label{rem:LambdaEmpty}
In order to simplify the notation in the theorems of this section, we assume that $\Lambda=\varnothing$ so that $I=K$. This means that all the bijections stated over $C$ are actually defined over every stratum of the natural stratification of $C$, i.e. given by the strata $C\cap\Do_R$ for $R\subset\Lambda$.
\end{rmk}

Let $\tilde L$ be the line bundle associated to the divisor $E$ and fix a section $\tilde s$ of $\tilde L$ such that $\tilde s^{-1}(0)=E$. We define 
$$\begin{array}{rcccl}
\tilde L_k\coloneq \tilde L^{-1}\otimes\sigma^*L_k&\text{ and }&\tilde s_k\coloneq\frac{\sigma^*s_k}{\tilde s} & \text{ for }&k\in K, \\ 
\tilde L_i\coloneq \tilde \sigma^*L_i &\text{ and }&\tilde s_i\coloneq\tilde \sigma^*s_i & \text{ for }&i\in I\setminus K.
\end{array}$$

Fix non-vanishing local sections $\ell_i$ of $L_i$, $i\in I$, and  $\tilde\ell$ of $\tilde L$. Then, for $i\in I$, let $\tilde\ell_i$  denote the section of $\tilde L_i$ defined by 
$\sigma^*\ell_i=\tilde\ell_i\cdot\tilde\ell$ if $i\in K$ and $\tilde\ell_i=\sigma^*\ell_i$ if $i\in I\setminus K$. 
These sections give local trivialisations (over Zariski open sets) of the corresponding line bundles and we use them, together with the section $s_k$, $\tilde s_k$, and $\tilde s$, to locally describe the components of $D$ and $\widetilde D$ as the zero loci of functions with values in $\mathbb C$. Indeed, the zero loci  of $g_k\coloneq s_k/\ell_k$, $\tilde g_k=\tilde s_k/\tilde\ell_k$, and $\tilde g=\tilde s/\tilde\ell$, are $D_k$, $\tilde D_k$, and $E$, respectively.  Note that $g_k\circ\sigma=\tilde g_k\cdot\tilde g$. 

We need a vector bundle $\LLL$ with a section $s$ so that $C= s^{-1}(0)$. Since $C$ is a subset of $D_K$, the sections $s_k$ already vanish on $C$.  Let $\LLC$ be a vector bundle defined on $M$ with a section $s_C$ such that the section $s= (s_C, s_k;k\in K)$ of the vector bundle 
$\LLL= \LLC\oplus\bigoplus_{k\in K} L_k$ is transverse to the zero section 
and defines $C$.  A suitable bundle $\LLC$ and section $s_C$ can always be chosen locally (for the Zariski topology). 

The blowing-up $\tilde M = Bl_C(M)$ of $C$ in $M$ can be described in terms of the section $s$ as 
\begin{equation}\label{eqn:Mtilde}
\tilde M=\text{ Closure of the image of } \,M\setminus C\longhookrightarrow 
\PP \LLL = \PP\Big (\LLC \oplus\bigoplus_{k\in K}L_k\Big), 
\end{equation} 
where the inclusion is given by $x\mapsto \left [ s \right] = \left[s_C (x):s_k(x),\,k\in K\right]$.  
In particular, \eqref{eqn:Mtilde} identifies the exceptional divisor $E$ of $\sigma$ with 
the projectivisation of $\LLL_{|C} $
\begin{align}\label{eqn:sectionsembedding}
\iso_s: E \xrightarrow{
   \,\smash{\raisebox{-0.65ex}{\ensuremath{\scriptstyle\sim}}}\,}
\PP \LLL_{|C} .
\end{align}
This identification depends on the choice of the section $s$.

\subsection{Normal bundles and the exceptional divisor}
In order to get rid of this dependence on the section $s$, we use the identifications in terms of normal bundles from Corollary \ref{cor:log-and-normal}.  

The normal bundle $\ta $ to $C$ in $M$, also denoted later by  $\ta_{M,C}$,  is the quotient bundle on $C$ defined by the exact sequence 
$$
\begin{tikzcd}
0\arrow[r]  & T C \arrow[r] 
  & TM_{|C} \arrow[r, "\pi_{C}"] & \arrow[r]
\ta    & 0
\end{tikzcd} \, .
$$
The differential of $s$ induces a vector bundle isomorphism $s': \ta \to \LLL_{|C}$.  For $x\in C$ we have the commutative diagram 
\begin{equation}\label{diag:name}
\begin{gathered}
\begin{tikzcd}
0\arrow[r]  & T_x C \arrow[r] 
 \arrow[d] & T_xM \arrow[r, "\pi_{C,x}"] \arrow[d, "T_x s"] & 
 \ta_{x}\arrow[d, "s'(x)"] \arrow[r] & 0\\
0 \arrow[r] & T_x M \arrow[r] & T_x \LLL \arrow[r, "\pi_{\LLL,x}"]& T_x \LLL_{x} 
= \LLL_{x} \arrow[r] & 0 \, ,
\end{tikzcd}
\end{gathered}
\end{equation}
where the map $T_x M\to T_x \LLL $ in the bottom row is the differential of the embedding $M \to \LLL$ of $M$ to the zero section of $\LLL$.

The exceptional divisor $E$ can be identified canonically with the projectivisation of the normal bundle 
\begin{align}\label{eqn:canon}
\can_E : E  \IsoTo 
\PP \ta 
. 
\end{align} 
The latter can be interpretted geometrically in terms of complex arc liftings.  Let $\gamma : (\C, 0) \to (M,x)$ be a complex arc such that $\gamma (t) \notin C$ for $t\ne 0$. We denote by $\tilde \gamma : (\C, 0) \to \tilde M$ the unique lifting of $\gamma$ to $\tilde M$ with respect to the blowing-up. If $\gamma'(0) \in T_x M \setminus T_xC$, then $\can_E (\tilde \gamma (0)) = [\pi_{C,x} \gamma'(0)] $.

The two embeddings \eqref{eqn:sectionsembedding} and \eqref{eqn:canon} are related by the formula 
\begin{align}\label{eqn:tocanon}
\iso_s =  \PP (s')  \circ \can_E ,
\end{align}
where $s'(x)$ is the linear isomorphism defined by \eqref{diag:name}.  
 
To show \eqref{eqn:tocanon}, take an arc  $\gamma : (\C, 0) \to (M,x)$, 
$v=\gamma'(0) \notin T_x C$.  Then the lifting of $\gamma$ to \eqref{eqn:Mtilde}, which we denote by $\tilde \gamma_s$,  is given, for $t\ne 0$, by $\tilde \gamma_s (t) = (\gamma(t), [s(\gamma (t))])$. Therefore, 
$$
\tilde \gamma_s (0) = \lim_{t\to 0} [s(\gamma (t))] = 
\lim_{t\to 0} [(s\circ\gamma)' (t))] =
[\pi_{\LLL ,x} 
(s\circ\gamma)' (0)]= [s'(x) (\pi_{C,x}(\gamma'(0)))] =  
\PP (s') ([\pi_{C,x} v]), 
$$ 
where the second equality states that the limits of secant and tangent directions coincide.  This completes the argument for \eqref{eqn:tocanon}.  

The isomorphism \eqref{eqn:canon} extends to a line bundle isomorphism as the normal bundle $\tilde \ta$ to $E$ in $\tilde M$ corresponds (canonically) to the tautological bundle on  $\PP \ta$.  Taking this identification into account, we associate to a nonzero vector $v\in \ta^*$ the point $[v]\in \PP \ta$ and a vector $\tilde v \in \tilde \ta^*_{[v]}$.  

 This map can again be described in terms of arc liftings. Let $\gamma(t)$ be an arc such that $\gamma(0)=x$ and $\pi_{C,x}(\gamma'(0))=v\in \ta^*$. Denote by $\tilde\gamma$ its lifting to $\tilde M$.  Then $\tilde\gamma(0)=[v]\in\PP(\ta_x)$ and $\pi_{E}(\tilde\gamma'(0))=\tilde v\in\tilde \ta^*_{[v]}$,  where 
$ \pi_{E} : T\tilde M_{|E} \to \tilde \ta$ denotes the projection.\\

Therefore, the following lemma holds.

\begin{lemma}\label{lem:normalid}
Let $\ta^*$ be the normal bundle of $C$ in $M$ without its zero section, and let $\tilde \ta^*$ be the normal bundle of $E$ in $\tilde M$ without its zero section. Then we have a canonical bijection 
$$\ta^*\ni v\mapsto([v],\tilde v)\in\tilde \ta^*. $$
\end{lemma}

We use the following notation for the normal bundles: $N_i=N_{M,D_i}$ and $N_J = \oplus_{i\in J} {N_i}_{|D_J}$.  The latter is  well defined on $D_J$ and equals the normal bundle of $D_J$, also denoted by $\ta_{M,{D_J}}$. Recall that $N^*_i$ denotes $N_i$ with the zero section removed, and by $N^\star_J$ we denote the fibre product of the $N^*_i$ over $\Do_J$, namely: $$N^\star_J \coloneq\fprod{\Do_J}{i\in J} N^*_i.$$
We use a similar notation for the normal fibres to the divisors in $\tilde M$. 

For $C\subset D_K\subset M$, if $\dim C<\dim D_K$ then we have an exact sequence of normal bundles on $C$:
\begin{align}\label{eqn:sequence_of_NB}
0\rightarrow N_{D_K,C} \xrightarrow[]{i} N_{M,C} \xrightarrow[]{\pi} {N_{K}}_{|C} \rightarrow0 .
\end{align}
Since $\ta_{{K}} = \oplus_{k\in K} {\ta_{k}}_{|D_K}$, we write the images of $\pi$ as 
$\pi(v)=(v_k,\,k\in K)$, where $v_k=\pi_k(v)$ denotes the k-th coordinate of $\pi(v)$.

Analogously to \eqref{eqn:comm_diagr_s} and \eqref{diag:name}, we have the following commutative diagram:
$$
\begin{tikzcd}
0 \arrow[r] & N_{D_K,C} \arrow[r,"i"] \arrow[d,"s_C'"',"\rotatebox{90}{$\sim$}"] & N_{M,C} \arrow[r,"\pi"] \arrow[d,"s'"',"\rotatebox{90}{$\sim$}"] & {N_{K}}_{|C}  \arrow[r] \arrow[d,"{(s_k',k\in K)}"',,"\rotatebox{90}{$\sim$}"]& 0 \\
0 \arrow[r] & \LLC_{|C} \arrow[r] & \displaystyle\LLC_{|C} \oplus\bigoplus_{k\in K}{L_k}_{|C} \arrow[r] & \displaystyle\bigoplus_{k\in K}{L_k}_{|C} 
\arrow[r] & 0 .
\end{tikzcd} 
$$
Since the vertical arrows are isomorphisms, the bottom sequence induces a splitting of the top one:
\begin{equation}\label{eqn:splitting}
\begin{tikzcd}
0 \arrow[r] & N_{D_K,C} \arrow[r,"i", yshift=.2em] & N_{M,C} \arrow[r,"\pi", yshift=.2em] \arrow[l,"i^*", yshift=-.2em] & {N_{K}}_{|C} \arrow[r] \arrow[l,"\pi^*", yshift=-.2em] & 0.
\end{tikzcd}
\end{equation}
This splitting is defined only on an open subset of $M$ where $s_C$ is defined, and, moreover, it depends on the choice of all sections $s_C$ and $s_k,\,k\in K$.  \\

The normal bundle of $D_k$, $k\in K$, and of its strict transform are related by the following.

 \begin{lemma}\label{lem:missing_link}
Over $E\cap \tilde D_k$, $k\in K$, the line bundle $\sigma^*N_k$ is canonically isomorphic to $\tilde N \otimes \tilde N_k$.
 \end{lemma}

 \begin{proof}
We have $\sigma^*L_k = \tilde L\otimes \tilde L_k$ and the derivatives of $s_k, \tilde s, \tilde s_k$  over $E\cap \tilde D_k$, similar to the maps $s'_i$ of \eqref{eqn:comm_diagr_s}, induce isomorphisms $\tilde N \simeq \tilde L$, $\tilde N \simeq \tilde L$, $\tilde N \simeq \tilde L$.  These isomorphisms depend on the choice of sections, but we show that the induced isomorphism $\sigma^*N_k\simeq \tilde N \otimes \tilde N_k$ is independent of this choice provided  $\sigma ^* s_k = \tilde s \tilde s_k$.

Recall, after \eqref{eqn:comm_diagr_s}, that the isomorphism  $N_k \simeq L_k$ is given by associating $w_k = s_k'(v_k)=D_x s_k(v_k)=\partial_{v_k} s_k\in L_{k,x}$ to $v_k \in N_{k,x}$. We show that for $w_k = s_k'(v_k)\in L_{k,x}$, $\tilde w = \tilde s'(\tilde v)\in \tilde L_{\tilde x}$, $\tilde w_k = \tilde s_k'(\tilde v_k)\in \tilde L_{k,\tilde x}$, the property $w_k=\tilde w \tilde w_k$ depends only on the vectors $v_k, \tilde v, \tilde v_k$ and not on the sections. (Here, for simplicity of notation, we identify $\sigma^* L_{k,x}$ and $L_{k,\tilde x}$.)
If $u_k$ is an analytic function not vanishing at $x$ then the section $u_ks_k$ associates to $v_k$ the vector $u_k(x) w_k$. This is because $\partial_{v_k}(u s_k)=u(x)\partial_{v_k} s_k+(\partial_{v_k} u_k)s_k(x)$ and $s_k(x)=0$. The same property holds for $\tilde u \tilde s$ and $\tilde u_k \tilde s_k$. Thus, if $u(x)=\tilde u(\tilde x)\tilde u_k(\tilde x)$, then the property  $w_k=\tilde w \tilde w_k$ is preserved as claimed.
 \end{proof}

\subsection{Algebraic case: $\Mmlog \simeq \Mmot$}\label{ss:motbu}
By Corollary \ref{cor:log-and-normal}, we may identify 
$$
\Mmlog_{|C} \simeq {N^\star_K}_{|C}.
$$
We stratify $E$ by $\tildeDoQ=E\cap (\bigcap_{q\in Q}\tilde D_q\setminus\bigcup_{k\in K\setminus Q}\tilde D_k)$, where $Q\subset K$; see \eqref{eqn:tildeJ}, \eqref{eqn:tildeD} and Remark \ref{rem:LambdaEmpty}. Thus the biggest stratum of $E$ is $\tilde D^\circ_{\tilde\varnothing}=E\setminus\cup_{k\in K}\tilde D_k$ for $Q=\varnothing$. If $Q=K$, then the stratum $\tilde D_{\tilde K}=E\cap (\bigcap_{k\in K}\tilde{D}_k)$ is  nonempty if and only if $\dim C< \dim D_K$.

Thus, for any $Q\subset K$, we have 
$$\tMmlog_{|\tildeDoQ} \simeq \tN^\star_{\tilde Q ~|\tildeDoQ}.
$$
Given $x\in C$, a preimage $\tilde x\in \sigma^{-1}(x)$ such that $\tilde x \in \tildeDoQ$, and $(\tilde v, \tilde v_q; q\in Q) \in {\tN^\star_{\tilde Q~|\tildeDoQ}}$,   
 by Lemma \ref{lem:normalid}, there is a unique vector  $v\in \ta_{M,C,x}$ corresponding to  $(\tilde x = [v], \tilde v)$. Let us denote it by $v=v(\tilde x, \tilde v)$.  Note that, since $\tilde x \in \tildeDoQ$, the coordinates $(v_q)_{q\in Q}$ of $\pi(v)\in N_K$ vanish and the coordinates $(v_k)_{k\in K\setminus Q}$ are non-zero. In other words, $v\in \pi^{-1} (N^\star_{K\setminus Q})$, where we consider $N^\star_{K\setminus Q}$ embedded in $N$ via the zero section on $N_q$ for $q\in Q$.  

 Let us now interpret the morphism $\sigma^\mlog : \tMmlog_{|\tildeDoQ} \to \Mmlog_{|C}$ in terms of normal bundles.

 \begin{prop}\label{prop:sigma_*}
The morphism   $\sigma^\mlog: \tMmlog_{|\tildeDoQ} \to \Mmlog_{|C}$ induces a map  $\sigma_N^\mlog : {\tN^\star_{\tilde Q~|\tildeDoQ}}   \to {N^\star_K}_{|C}$  given by 
\begin{align}\label{eqn:sigma_*}
\sigma_N^\mlog(\tilde x, \tilde v; \tilde v_q, q\in Q)=(\pi_k(v), {k\in K\setminus Q}; 
\tilde v\tilde v_q, q\in Q),
\end{align}
where $v=v(\tilde x, \tilde v)$.  
\end{prop}

 \begin{proof}
Recall that $\sigma^\mlog$ is defined by $\sigma^\mlog(\tilde \Phi) (g) 
= \tilde \Phi (g\circ \sigma)$. 
In order to determine $\Phi=\tilde \Phi (g\circ \sigma)$ at $x$, it suffices to consider only $g= g_k= s_k/\ell_k$ for $k\in K$ and to show that it corresponds, by Lemma \ref{lem:log-and-normal}, to $v_k= \pi_k(v)$ for $k\in K\setminus Q$ and 
$v_q = \tilde v\tilde v_q$ for $q\in Q$.  

For $q\in Q$, it follows from the definition of the product $\tilde v\tilde v_q$ given in 
the proof of Lemma \ref{lem:missing_link} and the 
formula $w_q=\Phi \left(\frac{s_q}{\ell_q}\right)\ell_q(x)$ showed in 
\eqref{eqn:computation-wi}.

For $k\in K\setminus Q$ we argue in the following manner. Let $\gamma(t)$ be the arc used in the proof of Lemma \ref{lem:normalid}.  Recall that $\gamma(0)=x$ and $\pi_{C,x}(\gamma'(0))=v$. The lifting $\tilde\gamma$ of $\gamma$ satisfies $\tilde\gamma(0)=\tilde x$ and $\pi_{E}(\tilde\gamma'(0))=\tilde v\in\tilde \ta^*_{\tilde x}$.  Differentiating 
$g_k(\gamma (t)) = \tilde g(\tilde \gamma (t)) \tilde g_k(\tilde \gamma (t)) $ and setting $t=0$, we get 
$$
\partial_{v_k} g_k (x)= \partial_{v} g_k (x) = \tilde g_k (\tilde x) \partial_{\tilde v} \tilde g (\tilde x) .
$$
Since $\tilde g_k (\tilde x) = \tilde \Phi (\tilde g_k)$ and  $\partial_{\tilde v} \tilde g (\tilde x) = \tilde \Phi (\tilde g)$, this shows the claim.  
\end{proof}

 For $Q\subsetneq K$, we define two maps: 
 $$H_{\tilde Q}^\mlog :  {\tN^\star_{\tilde Q~|\tildeDoQ}} \to 
 \pi^{-1} ({N^\star_{K\setminus Q}}_{|C}) \times_C
  {N^\star_{Q}}_{|C},$$
  given by $H_{\tilde Q}^\mlog (\tilde x, \tilde v; \tilde v_q, q\in Q) 
  = (v(\tilde x, \tilde v); v_q=\tilde v \tilde v_q, q\in Q)$, where $\pi$ is given by \eqref{eqn:sequence_of_NB}, and 
$$
h_{\tilde Q}^\mlog : \pi^{-1} ({N^\star_{K\setminus Q}}_{|C}) \times_C
  {N^\star_{Q}}_{|C} \to {N^\star_{K\setminus Q}}_{|C} \times_C
  {N^\star_{Q}}_{|C} = {N^\star_K}_{|C} ,$$
 induced by $\pi$; that is, $h_{\tilde Q}^\mlog (v;  v_q, q\in Q) = (\pi_k (v), k\in K\setminus Q; v_q, q\in Q)$. 

\begin{thm}\label{thm:details-blow-up-mot}
For $Q\subsetneq K$, the map $H_{\tilde Q}^\mlog $ is a bijection and the following diagram commutes  \\
$$
\begin{tikzcd}
\tMmlog_{|\tildeDoQ} \ar[r,"\simeq"] \ar[rd, "\sigma^\mlog" '] & {\tN^\star_{\tilde Q~|\tildeDoQ}} \ar[rd, "\sigma_N^\mlog"] \ar[rr, "H_{\tilde Q}^\mlog"] &               
&  \pi^{-1} ({N^\star_{K\setminus Q}}_{|C}) \times_C {N^\star_{Q}}_{|C}\arrow[ld, "h_{\tilde Q}^\mlog" ] \\
 & \Mmlog_{|C}\ar[r,"\simeq"]                                                         & {N^\star_K}_{|C}  &
\end{tikzcd} .
$$
\end{thm}

\begin{proof}
The commutativity of the diagram is a consequence of the formula 
\begin{align}
\sigma^\mlog (\tilde x, \tilde v; \tilde v_q, q\in Q)=((\pi_k(v))_{k\in K\setminus Q},v_q=\tilde v_q\cdot\tilde v),
\end{align}
which follows from Proposition \ref{prop:sigma_*}. 

To see that $H_{\tilde Q}^\mlog $ is a bijection, we construct its inverse. Let $(v; v_q, q\in Q) \in 
\pi^{-1} ({N^\star_{K\setminus Q}}_{|C}) \times_C
  {N^\star_{Q}}_{|C}$.  Then $\pi(v)=(v_k,\,k\in K)$ and the coordinates $v_k=\pi_k(v)$ for $k\in K\setminus Q$ are all  non-zero. 
By Lemma \ref{lem:normalid}, $v$ gives  $([v],\tilde v)\in\tilde \ta^*$.  Denote $\tilde x\coloneq [v]$.  Because $\tilde x \in \tildeDoQ$, the 
$q$-coordinates $\pi(v)$, for $q\in Q$, all vanish. That is why we have to chose different numbers $(v_q)_{q\in Q} \in {N^\star_{Q}}_{|C}$, independent of $v$, in order for $H_{\tilde Q}^\mlog $ to be bijective.  
Now $(v; v_q, q\in Q)$ gives $(\tilde x, \tilde v; \tilde v_q, q\in Q) $ by the identity $(v(\tilde x, \tilde v); v_q=\tilde v \tilde v_q, q\in Q)$.  
\end{proof}

Note that  each fibre of $h_{\tilde Q}^\mlog$ is a fibre of $\pi$.   
 This allows us to trivialise $h_{\tilde Q}^\mlog$ locally for the Zariski topology.  
The trivialising map  
$$\tr_{\tilde Q}^\mlog : \pi^{-1} ({N^\star_{K\setminus Q}}_{|C}) \times_C
  {N^\star_{Q}}_{|C}\to{N^\star_K}_{|C}  \times_C N_{D_K,C}$$ 
  is defined by 
  $$\tr_{\tilde Q}^\mlog (v; v_q, q\in Q) = 
  (\pi_k(v), k\in K\setminus Q; v_q, q\in Q,i^* (v)),
  $$
 where $i^*$ is defined by the splitting \eqref{eqn:splitting}. 
 Recall that, in general, this splitting is only defined locally. 

\begin{cor}\label{cor:details-blow-up-mot1}
Let $Q\subsetneq K$.  The composition map 
$S_{\tilde Q}^\mlog = \tr_{\tilde Q}^\mlog \circ H_{\tilde Q}^\mlog$, defined over 
a Zariski open $U\subset C$ over which the splitting \eqref{eqn:splitting} exists,
is a bijection making the following diagram commutative 
$$
\begin{tikzcd}
\tMmlog_{|\tildeDoQ\cap {\sigma^{-1}(U)} } \ar[r,"\simeq"] \ar[rd, "\sigma^\mlog" '] & {\tN^\star_{\tilde Q~|{\tildeDoQ\cap {\sigma^{-1}(U)} }}} \ar[rd, "\sigma_N^\mlog"]  \ar[rr, "S_{\tilde Q}^\mlog"] & &  {N^\star_K}_{|U}  \times_C {N_{D_K,C}}_{|U}  \arrow[ld, "pr_1" ] \\
& \Mmlog_{|U} \ar[r,"\simeq"] & {N^\star_K}_{|U} & 
\end{tikzcd} .
$$ 
\end{cor}

\begin{rmk}
Note that the target space of $S_{\tilde Q}^\mlog$ is independent of $Q$. We will see in the proof of the independence of the resolution of the image of $\Mmlog $ in the Grothendieck ring , Theorem  \ref{thm:invSf}, how these terms, equipped with the correct signs, cancel out.   
\end{rmk}

If $Q=K$, then the situation is even simpler since, in this case $\pi (v)=0$, and, therefore, by \eqref{eqn:sequence_of_NB}, $v \in N^*_{D_K,C} $.  In this case we define
 $$S_{\tilde K}^\mlog : {\tN^\star_{\tilde K~|\tildeDoK}} \to  
  {N^\star_K}_{|C} \times_C N^*_{D_K,C},$$
  by $S_{\tilde K}^\mlog (\tilde x, \tilde v; \tilde v_k, k\in K) 
  = (v_k=\tilde v \tilde v_k, k\in K; v(\tilde x, \tilde v))$. Then we obtain easily

\begin{prop}\label{prop:details-blow-up-mot2}
$S_{\tilde K}^\mlog $ is a bijection making the following diagram commutative  \\
$$
\begin{tikzcd}
\tMmlog_{|\tildeDoK}  \ar[r,"\simeq"]  \ar[rd, "\sigma^\mlog" '] & {\tN^\star_{\tilde K~|\tildeDoK}} \ar[rd, "\sigma_N^\mlog"]
  \ar[rr, "S_{\tilde K}^\mlog"]& & {N^\star_K}_{|C}  \times_C N^*_{D_K,C}
 \arrow[ld, "pr_1" ] \\
& \Mmlog_{|C} \ar[r,"\simeq"] & {N^\star_K}_{|C} & 
\end{tikzcd} .
$$ 
\end{prop}

\subsection{Topological case: $\Mlog \simeq \Mtop$}\label{ss:top} 
First, dividing by the action of powers of $\Rp$, we obtain $\log$ versions of Theorem \ref{thm:details-blow-up-mot}, Corollary \ref{cor:details-blow-up-mot1}, and Proposition \ref{prop:details-blow-up-mot2}.  
Then we show in Corollary \ref{cor:glued-blow-up-top1} that, in this case, the local trivialisations, given stratum by stratum in Corollary \ref{cor:details-blow-up-top1} and Proposition \ref{prop:details-blow-up-top2}, can be glued by a simple formula to one local trivialisation that is a homeomorphism.

By Corollary \ref{cor:log-and-normal},
$$
\Mlog_{|C} \simeq \Sph^\star({N_K}_{|C}), \quad\text{ and }\quad \tMlog_{|\tildeDoQ} \simeq \Sph^\star({{\tN}_{\tilde Q~|\tildeDoQ}}).
$$

Dividing by $\Rp$, the bijection of Lemma \ref{lem:normalid} induces a bijection 
$$
\Sph(\tilde N) \ni (\tilde x ,\tilde \theta) = ([v],\theta(\tilde v)) \leftrightarrow \theta(v) \in \Sph(N). $$ 
 Because $\theta(v)$ depends only on 
 $ (\tilde x=[v] ,\tilde \theta = \theta(v))$, we also denote it by  
$\theta (\tilde x, \tilde \theta)$ and, similarly, $\theta(\pi_k(v))$ by  
$\theta_k (\tilde x, \tilde \theta)$.   

 \begin{prop}\label{prop:logsigma_*}
The morphism   $\sigma^\log : \tMlog_{|\tildeDoQ} \to  \Mlog_{|C}$ induces a map $\sigma_N^\log:\Sph^\star({\tN_{\tilde Q~|\tildeDoQ}})   \to \Sph^\star({N_K}_{|C})$ given by 
\begin{align}\label{eqn:logsigma_*}
\sigma_N^\log (\tilde x, \tilde \theta; \tilde \theta_q, q\in Q)=(  
\theta_k (\tilde x, \tilde \theta), {k\in K\setminus Q}; 
\tilde \theta\tilde \theta_q, q\in Q) .
\end{align}
\end{prop}

  For $Q\subsetneq K$, we define two maps: 
 $$
 H_{\tilde Q}^\log :  \Sph^\star({\tN_{\tilde Q~|{\tildeDoQ}}}) \to \Sph (\pi^{-1} ({N^\star_{K\setminus Q}}_{|C})) \times_C
  \Sph^\star({N_{Q}}_{|C}),$$
  given by $H_{\tilde Q}^\log (\tilde x, \tilde \theta; \tilde \theta_q, q\in Q) 
  = (\theta(\tilde x, \tilde \theta); \theta_q=\tilde \theta \tilde \theta_q, q\in Q)$, and, 
\begin{align}\label{eqn:hlogKQ}
h_{\tilde Q}^\log : \Sph (\pi^{-1} ({N^\star_{K\setminus Q}}_{|C})) \times_C
 \Sph^\star({N_{Q}}_{|C}) \to \Sph^\star({N_{K\setminus Q}}_{|C}) \times_C
  \Sph^\star({N_{Q}}_{|C}) = \Sph^\star({N_K}_{|C}),
  \end{align}
defined by  
$h_{\tilde Q}^\log (\theta(v);  \theta_q, q\in Q) = (\theta(\pi_k (v)), k\in K\setminus Q; \theta_q, q\in Q)$. 

\begin{thm}\label{thm:details-blow-up-log}
For $Q\subsetneq K$, $H_{K,\tilde Q}^\log $ is a bijection and the following diagram commutes  \\
$$
\begin{tikzcd}
 \tMlog_{|\tildeDoQ }  \ar[r,"\simeq"] \ar[rd, "\sigma^\log" '] & \Sph^\star({\tN_{\tilde Q~|\tildeDoQ}}) 
 \ar[rd, "\sigma_N^\log"] \ar[rr, "H_{\tilde Q}^\log"]   & &  \Sph({(\pi^{-1} (N^\star_{K\setminus Q}))}_{|C}) \times_C
  \Sph^\star({N_{Q}}_{|C})
 \arrow[ld, "h_{\tilde Q}^\log" ] \\
&\Mlog  \ar[r,"\simeq"]  & \Sph^\star({N_K}_{|C}) & 
\end{tikzcd} .
$$
\end{thm}

Let us compute the fibres of $h_{\tilde Q}^\log$. For this, it suffices to look at the map 
$$\beta : 
\Sph (\pi^{-1} ({N^\star_{K\setminus Q}}_{|C}))  \to \Sph^\star({N_{K\setminus Q}}_{|C})$$
 given by the first coordinates of \eqref{eqn:hlogKQ}. Consider the following diagram where all the normal bundles are restricted to $C$:
$$
\begin{tikzcd}
0 \arrow[r] & N_{D_K,C} \arrow[r] & N_{M,C} \arrow[r, "\pi"] & 
N_K  \arrow[r] & 0 \\
& & \pi^{-1} ( N_{K\setminus Q}^\star ) 
\arrow[hookrightarrow]{u} \arrow[r] 
 \arrow[d] & N_{K\setminus Q}^\star 
\arrow[hookrightarrow]{u} \arrow[d, "\alpha"] \\
& &  \Sph  (\pi^{-1} (N_{K\setminus Q}^\star ) ) 
\arrow[r, "\beta"] &  \Sph^\star(N_{K\setminus Q}^\star)
\end{tikzcd}
$$
The map $\alpha$ is given by dividing each $N_k^*$ by $\Rp$.  Therefore 
every fibre of $\alpha$ can be identified with 
$\fprod {C} {{k\in K\setminus Q}} N_{k,>0}$ via the projections $N_k^* \to (N_k^* \mod S^1) \coloneq N^*_{k,>0}$.  
Using the spliting in \eqref{eqn:splitting}, we may identify 
every fibre of $\beta$, and hence every fibre of $h_{\tilde Q}^\top$, with 
$\Sph\big(N_{D_K,C}\times_C \displaystyle\fprod{C} 
{k\in K\setminus Q} N_{k,>0}\big )_{|C} $.    
This induces the map  
$$
\tr_{\tilde Q}^\log : \Sph({(\pi^{-1} (N^\star_{K\setminus Q}))}_{|C}) \times_C
  \Sph^\star({N_{Q}}_{|C})\to \Sph^\star({N_K}_{|C})  \times_C \Sph\big(N_{D_K,C}\times_C 
  \displaystyle\fprod{C} 
{k\in K\setminus Q} N_{k,>0}\big )_{|C},$$ 
   which trivialises $h_{\tilde Q}^\log$ 
locally for the  Zariski topology and is defined by 
$$
\tr_{\tilde Q}^\log (\theta(v);  \theta_q, q\in Q) = 
(h_{\tilde Q}^\log (\theta(v),   \theta_q, q\in Q), 
(i^* (v); r_k(v) , k\in K\setminus Q) \mod \Rp ),$$
where, recall,  $r_k(v)\coloneq\pi_k(v) \mod S^1$.

\begin{cor}\label{cor:details-blow-up-top1}
Let $Q\subsetneq K$. Over a Zariski open subset $U\subset C$ for which the splitting from \eqref{eqn:splitting} exists, the composition map $S_{\tilde Q}^\log = \tr_{\tilde Q}^\log \circ H_{\tilde Q}^\log$ is a bijection making the following diagram commutative over $U$:
$$
\begin{tikzcd}
\tMlog_{|\tildeDoQ }  \ar[r,"\simeq"] \ar[rd, "\sigma^\log" '] & \Sph^\star ({\tN_{\tilde Q~|\tildeDoQ}})\ar[rd, "\sigma_N^\log"] 
  \ar[rr, "S_{\tilde Q}^\log"]& &
\Sph^\star (N_K)  \times_U \Sph\big(N_{D_K,C}\times_U  
  \displaystyle\fprod{U} 
{k\in K\setminus Q} N^*_{k,>0}\big )
 \arrow[ld, "pr_1" ] \\
&\Mlog  \ar[r,"\simeq"] & \Sph^\star (N_K) & 
\end{tikzcd} .
$$ 
\end{cor}

Similar to the algebraic case, if $Q=K$, then the situation becomes much simpler. In this case, we define
 $$S_{\tilde K}^\log : \Sph^\star({\tN_{\tilde K~|\tildeDoK}}) \to 
\Sph^\star({N_K}_{|C})  \times_C \Sph(N_{D_K,C}) ,$$
  by $S_{\tilde K}^\log (\tilde x, \tilde \theta; \tilde \theta_k, k\in K) 
  = (\theta_k=\tilde \theta \tilde \theta_k, k\in K; \theta(\tilde x, 
  \tilde \theta))$.

\begin{prop}\label{prop:details-blow-up-top2}
$S_{\tilde K}^\log $ is a bijection making the following diagram commutative  \\
$$
\begin{tikzcd}
\tMlog_{|\tildeDoK }  \ar[r,"\simeq"] \ar[rd, "\sigma^\log" '] & \Sph^\star({\tN_{\tilde K~|\tildeDoK}}) \ar[rr, "S_{\tilde K}^\log"]\ar[rd, "\sigma^\log"]& & \Sph^\star({N_K}_{|C})  \times_C \Sph(N_{D_K,C})  \arrow[ld, "pr_1" ] \\
& \Mlog_{|C} \ar[r,"\simeq"]&\Sph^\star({N_K}_{|C})  & 
\end{tikzcd} .
$$ 
\end{prop}

The maps $S_{\tilde Q}^\log$, for $Q\subset K$, glue together to a map defined on $\tMlog_{|E}$. We are going to describe its topological structure in order to study the topological structure of $\sigma^\log$.  
This is particularly simple over an open subset 
$U\subset C$ over which the splitting from \eqref{eqn:splitting} holds.  
In this case, the inverses of $S_{\tilde Q}^\log$ glue together over $U$ to the map 
$$
(S^\log)^{-1} : \Sph^\star(N_K)  \times_U \Sph\big(N_{D_K,C}\times_U  
  \displaystyle\fprod{U} 
{k\in K} N_{k,\ge 0}\big) \to \bigcup_Q \Sph^\star({\tN_{\tilde Q~|\tildeDoQ}}) 
\simeq \tMlog_{|E} , 
$$
where $(\theta_k, k\in K, 
\theta (v_N; r_k , k\in K)) $  is mapped to the vector 
$v= i(v_N) + \pi^*(\sum_k r_k\theta _k)$. By Lemma \ref{lem:normalid}, 
the vector $v$ defines $\tilde x=[v]$ and $\tilde \theta = \theta (\tilde v)$ ($v$ is defined only modulo $\Rp$ but so is $\tilde \theta$).  Let $Q= \{k\,:\,r_k=0\}$. 
Then $\tilde x \in \tildeDoQ $ and we set $(\tilde \theta_q = \tilde \theta^{-1} 
\theta_q; q\in Q)\in \Sph^\star({\tN_{\tilde Q~|\tildeDoQ}})$.  It is easy to see that 
it is the inverse of $S_{\tilde Q}^\log$.  

\begin{cor}\label{cor:glued-blow-up-top1}
Over a Zariski open subset $U\subset C$ for which the splitting from \eqref{eqn:splitting} exists, the map $S^\log $ is a homeomorphism making the following diagram commutative over $U$:
$$
\begin{tikzcd}
\tMlog_{|E} 
 \ar[rd, "\sigma^\log" ']  \ar[rr, "S^\log"]& &  
\Sph^\star(N_K)  \times_U \Sph\big(N_{D_K,C}\times_U  
  \displaystyle\fprod{U} 
{k\in K} N_{k,\ge 0}\big )
 \arrow[ld, "pr_1" ] \\
& \Mlog \simeq \Sph^\star(N_K) & 
\end{tikzcd} .
$$ 
\end{cor}

By Corollary \ref{cor:glued-blow-up-top1}, each fibre of $\sigma^\log$ is of the form $\Sph\big(\C^m\times \displaystyle\prod_{k\in K}\R_{\ge0}\big)$, $m=\dim D_K -\dim C$, which is the topological join of $S^{2m-1}$ and the simplex $\Delta^{|K|-1}= (\R^K_{\ge0} \setminus \{0\}) / \Rp$. Since $K\neq\varnothing$, each fibre is contractible. \\

\begin{cor}\label{cor:homotopyequiv}
$\sigma: \tMtop\to \Mtop$ is a homotopy equivalence.
\end{cor}

\begin{proof}
Since the fibres are contractible, $\sigma^\top$ is a weak homotopy equivalence by a theorem of Smale \cite{Sma57}. Hence, by the Whitehead theorem, it is a homotopy equivalence since the involved objects are CW-complexes.
\end{proof}

\subsection{Complete case: $\Mclog \simeq \Mcpl$}\label{sect:compl}

We extend the above constructions to $\tMclog\to \Mclog$. To do so, we break $\tMclog_{|E}$ into its part at infinity and its interior part, defined as
$$\tMclog_{E,\infty}=\left\{(x,\varphi,\psi)\in\tMclog\,:\,x\in E,\,\psi(\tilde g)=\infty\right\}\text{ and }$$ 
$$\tMclog_{E, \inte}=\left\{(x,\varphi,\psi)\in\tMclog\,:\,x\in E,\,\psi(\tilde g)\in\Rp\right\},$$
respectively; so that
$$\tMclog_{|E}=\tMclog_{E,\infty}\sqcup \tMclog_{E,\inte}.$$
The first set corresponds to the boundary of the infinitesimal tubular neighbourhood of $E$ on $\Tilde M$; that is, the horizontal straight and dashed red lines on the right-hand side diagrams of Examples \ref{ex:bu} and \ref{ex:bu2}. The second set -- that is, the complement of the first -- corresponds to the interior of this neighbourhood.

For $Q\subset K$, we set 
$$\Mclog_{Q,int}=\left\{(x,\varphi,\psi)\in\Mclog\,:\,\forall k\in K\setminus Q,\,\psi(g_k)\in\Rp\right\} .$$

\begin{thm}\label{thm:details-blow-up-cpl1}
For every $Q\subsetneq K$, there is a canonical bijection 
$S_{Q,\inte}^\clog$ which extends  $S_{\tilde Q}^\mlog$ and makes the following diagram commutative locally over $C$:
$$
\begin{tikzcd}
\tMclog_{E,\inte|\tildeDoQ}  \ar[rd,  "\sigma^\clog" ']  \ar[rr, "S_{Q,\inte}^\clog"]& & 
 {\Mclog_{Q,int}}_{|C}\times_C {N_{D_K,C}}
 \arrow[ld, "\pr_1" ] \\
& \Mclog_{|C} & 
\end{tikzcd}
$$
If $Q=K$ and $\dim C < \dim D_K$, then there is a canonical bijection $S_{\tilde K,\inte}^\clog$, extending  $S_{\tilde K}^\mlog$,  making the following diagram commutative.
$$
\begin{tikzcd}
\tMclog_{E,\inte|\tildeDoK}  \ar[rd, "\sigma^\clog" '] \arrow[rr, "S_{\tilde K,\inte}^\clog"]
& & \Mclog_{|C}\times_C N^*_{D_K,C}   \arrow[ld, "\pr_1" ] \\
& \Mclog_{|C} & 
\end{tikzcd} $$
\end{thm}

\begin{proof}
The proof is virtually identical to the proof of Corollary \ref{cor:details-blow-up-mot1}.  
\end{proof}

Similarly, we study $\tMclog_{E,\infty}$ by extending the topological case.  Since $\tilde \psi (\tilde g)=\infty $, all $\psi (g_k) =\infty$ for $k\in K$ on the image 
 $\sigma^\clog (\tMclog_{E,\infty}) \subset \Mlog$.  Consider, as in Corollary \ref{cor:details-blow-up-top1},  
the sphere bundle 
$$
\Sph\Big(N_{D_K,C}\oplus\displaystyle\bigoplus_{k\in K\setminus Q}
N^*_{k,\ge 0}\Big) \xrightarrow{\sigma_Q} M .
$$
Note that $\Sph\Big(N_{D_K,C}\oplus\displaystyle\bigoplus_{k\in K\setminus Q}
N^*_{k,\ge 0}\Big)$ projects to $\PP \big(N_{D_K,C}\oplus\displaystyle\bigoplus_{k\in K\setminus Q}
N_{k}\big) \subset \PP \big(N_{D_K,C}\oplus\displaystyle\bigoplus_{k\in K}
N_{k}\big)$. For $q\in Q$, denote by $\sigma_Q^* \Big(\bigoplus_{q\in  Q} \tilde N_{q,\ge 0} \Big)$ the pull-back of 
$\bigoplus_{q\in  Q} \tilde N_{q,\ge 0} $ to  $\Sph\Big(N_{D_K,C}\oplus\displaystyle\bigoplus_{k\in K\setminus Q}
N^*_{k,\ge 0}\Big)
$.

\begin{thm}\label{thm:details-blow-up-cpl2}
For every $Q\subset K$,  there is a natural bijection $S_{Q,\infty} ^\clog$  making the following diagram commutative
\begin{equation*}
\begin{gathered}
\begin{tikzcd}
\tMclog_{E,\infty} \ar[d, " " '] \arrow[rr, "S_{Q, \infty}^\clog"] & & \Mlog_{|C}\times_C \sigma_Q^* \Big(\bigoplus_{q\in  Q} \tilde N_{q,\ge 0} \Big)  \arrow[d, " "] \\
\tMlog_{|\tildeDoQ} \ar[rd, "\sigma^\log"  '] \arrow[rr, "S_{\tilde Q}^\log"] & & \Mlog_{|C}\times_C\Sph\Big(N_{D_K,C}\oplus\displaystyle\bigoplus_{k\in K\setminus Q} N^*_{k,\ge 0}\Big) \arrow[ld, "\pr_1"] \\
& \Mlog_{|C }&
\end{tikzcd}
\end{gathered} ,
\end{equation*}
where the first vertical arrow is given by forgetting $\psi$.
\end{thm}

\begin{proof}
It suffices to find a formula for the second component of  $S_{Q, \infty}^\clog$.  It is given by $\hat r_q=(\tilde s_q')^{-1}\psi(\tilde g_q)\tilde\ell_{q,\ge 0}\in\tilde N_{q,\ge 0}$ for all $q\in Q$. To show that it injective, we note that $\psi (\tilde g_k)$ for $k\in K\setminus Q$ is uniquely defined since $\psi(\tilde g_k)\tilde\ell_{k,\ge 0}=\tilde s_{k,\ge 0}(\tilde x)$. It is easy to check that $S_{Q,\infty}^\clog $ is surjective.

\end{proof}

\section{The motivic Milnor fibre}\label{sec:localcase}
In this section, we consider $f:(X,x_0)\to(\mathbb C,0)$, i.e. the local case without the assumption that $f$ is normal crossing. Then, there exists a resolution of $f$ that is a finite sequence of blowings-up with smooth algebraic centres $\mu:M\to X$ such that the exceptional divisor has normal crossings only, and such that $\mu^{-1}(x_0)=\cup_{i\in I} D_i$ is the union of some components of the exceptional divisor. Then, $\tilde f(x)\coloneq f\circ\mu(x)=u(x)\prod_{i\in I}s_i(x)^{N_i}$ satisfies the assumptions from the set-up (see \eqref{eqn:globalformula}), where $s_i$ is a regular section of a line bundle $p_i:L_i\to M$ such that $D_i$ is the reduced variety defined by $s_i$.

By the weak factorisation theorem \cite{Wlo03}, two such resolutions can be related by a sequence of blowings-up and blowings-down. Therefore, the spaces $\Mmlog$, $\Mlog$, and $\Mclog$ induced by $\tilde f$ are well-defined up to equivalences induced by such a single blowing-up, as expressed in the theorems of Section \ref{sec:bl}.

We set 
$$\Mmotloc\coloneq\Mmot\cap\pr^{-1}\left(\mu^{-1}(x_0)\right),$$
where, recall, $\pr: \Mmot \to M$ denotes the projection. Then 
$$\Mmotloc=\bigsqcup_{\varnothing\neq J\subset I}L^\star_{J,x_0}\ ,$$ where $$L^\star_{J,x_0} \coloneq \fprod{\Do_J}{i\in J}{L^*_i}_{|{\Do_J\cap (\mu\circ\pr)^{-1}(x_0)}}.$$
Again, we have a log-geometric characterisation $\Mmlogloc\simeq\Mmotloc$, where
$$
\Mmlogloc=\left\{(x,\Phi)\,:\,\mu(x) = x_0,\,\Phi\in\Hom_{\mon}(\M_x,\C^*),\,
\forall g\in\OO_x^*,\,\Phi(g)=g(x)\right\}.$$

We define $\tilde f_J:\C^*\acts L_{J,x_0}^\star\to \C^*$ analogously to \eqref{def:f_J} or, equivalently, by applying the functoriality of $\Mmlogloc$ to $\tilde f:(M,D)\to(\mathbb C,0)$. Then we define the local motivic Milnor fibre of $f$ at $x_0$ by
$$
\displaystyle\mathcal S_{f,x_0}\coloneq-\sum_{\varnothing\neq J\subset I}(-1)^{|J|}\left[\tilde f_J:\C^*\acts L_{J,x_0}^\star\to \C^*\right]\in\K(\Var_{\C^*}^{\C^*}).
$$

It follows from \cite{DL98} that, as an element of $\M_{\C^*}^{\C^*}\coloneq\K(\Var_{\C^*}^{\C^*})\left[\LL^{-1}\right]$, $\mathcal S_{f,x_0}$ does not depend on the choice of the resolution. Indeed, the image of $\mathcal S_{f,x_0}$ in $\M_{\C^*}^{\C^*}$ is given by the formal limit $$\displaystyle\mathcal S_{f,x_0}=-\lim_{T\to\infty}Z_{f,x_0}(T),$$ where $$Z_{f,x_0}(T)\coloneq\sum_{n\geq 1}\left[\ac_f:\X_{n,x_0}(f)\to\C^*\right]\LL^{-nd}T^n\in\M_{\C^*}^{\C^*}\llbracket T\rrbracket$$ is the local motivic zeta function of $f$ at $x_0$. This formal limit is defined using a rational expression of the zeta function. The coefficients of $Z_{f,x_0}(T)$ are defined by
$$\X_{n,x_0}(f)\coloneq\left\{\varphi\in\mathcal L_n(X),\,\varphi(0)=x_0,\,\ord_t f(\varphi)=n\right\},$$ where $\ac_f(\varphi)=\ord_t(f\circ\varphi)$ is the \emph{angular component} mapping and $\L_n(X)$ denotes the space of $n$-jets on $X$, namely the reduced and separated scheme of finite type given by
$$\L_n(X)\coloneq X\!\left(\C\llbracket t\rrbracket/t^{n+1}\right)=\Hom_{\C\mathrm{-sch}}\left(\Spec\C\llbracket t\rrbracket/t^{n+1},X\right).$$

The following theorem states that $\mathcal S_{f,x_0}$ is well defined as an element of $\K(\Var_{\C^*}^{\C^*})$ and not merely of $\M_{\C^*}^{\C^*}$.

\begin{thm}\label{thm:invSf}
The motive $\mathcal S_{f,x_0}\in\K(\Var_{\C^*}^{\C^*})$ does not depend on the choice of the resolution $\mu$.
\end{thm}

\begin{proof}
By the weak factorisation theorem \cite{Wlo03}, it is enough to study the composition of $\mu$ with a single blowing-up $\sigma$ along a nonsingular subvariety of $D$ that is supposed to have normal crossings with $D$.

We use the notation of Section \ref{sec:bl}. For $Q\subset K$, denote by $(\tilde f\circ\sigma)_{\tilde Q}:\C^*\acts {\tilde L}^\star_{\tilde Q,x_0}\to \C^*$ the functions induced by $\tilde f\circ\sigma$.

Note that the bijections $S_{\tilde Q}^\mlog$ in Corollary \ref{cor:details-blow-up-mot1} and Proposition \ref{prop:details-blow-up-mot2} are equivariant isomorphisms with respect to the diagonal action on $\tMmot_{|\tildeDoQ}$ and the action on $\Mmot_{|C}$ given by $\lambda \cdot v_k=\lambda v_k$ if $k\in K\setminus Q$, and by $\lambda \cdot v_q=\lambda^2 v_q$ if $q\in Q$.

Therefore, for $Q\subsetneq K$, we get
$$\left[(\tilde f\circ\sigma)_{\tilde Q}:\C^*\acts {\tilde L}^\star_{\tilde Q,x_0}\to\C^*\right]=\left[f_K\circ\pr_1:\C^*\acts L_{K,x_0}^\star\times N_{D_K,C}\to\C^*\right],$$
and, if $Q=K$, we get
\begin{align*}\left[(\tilde f\circ\sigma)_{\tilde K}:\C^*\acts {\tilde L}^\star_{\tilde K,x_0}\to\C^*\right]&=\left[f_K\circ\pr_1:\C^*\acts L_{K,x_0}^\star\times N_{D_K,C}^*\to\C^*\right]\\&=\left[f_K\circ\pr_1:\C^*\acts L_{K,x_0}^\star\times N_{D_K,C}\to\C^*\right]-\left[f_K:\C^*\acts L_{K,x_0}^\star\to\C^*\right],\end{align*} where the last equality comes from additivity.

Note that the above equalities do not depend on $Q\subset K$; see Proposition \ref{prop:actions}.

Finally,
\begin{align*}
 &\sum_{Q\subset K}(-1)^{|Q|+1}\left[(\tilde f\circ\sigma)_{\tilde Q}:\C^*\acts {\tilde L}^\star_{\tilde Q,x_0}\to\C^*\right] \\
=&\sum_{Q\subset K}(-1)^{|Q|+1}\left[f_K\circ\pr_1:\C^*\acts L_{K,x_0}^\star\times N_{D_K,C}\to\C^*\right]+(-1)^{|K|}\left[f_K:\C^*\acts L_{K,x_0}^\star\to\C^*\right]\\
=&(-1)^{|K|}\left[f_K:\C^*\acts L_{K,x_0}^\star\to\C^*\right]
\end{align*}
since $\displaystyle{\sum_{Q\subset K}(-1)^{|Q|}=\sum_{n=0}^{|K|} \binom {|K|} {n}(-1)^n=0}$.
\end{proof}

\begin{rmk}\label{rmk:speculation2}
The above proof shows that the coefficients $(-1)^{|J|}$ appearing in the formula for $\Sf$ are necessary for the motivic Milnor fibre to be independent of the choice of the resolution (see also Remark \ref{rmk:speculations}).
\end{rmk}

\begin{rmk}
Similar remarks can be made in the topological and complete cases. Set $$\Mtoploc\coloneq\Mtop\cap\pr^{-1}\left(\mu^{-1}(x_0)\right).$$
We have a log-geometric characterisation $\Mlogloc\simeq\Mtoploc$ where
$$\Mlogloc=\left\{(x,\varphi)\,:\; \mu(x)=x_0,\,\varphi\in\Hom_{\mon}(\M_x, S^1),\,\forall g\in\OO_x^*,\,\varphi(g)=\frac{g(x)}{|g(x)|}\right\}.$$
By functoriality, $\tilde f$ induces a continuous map $\sign\tilde f:\Mlogloc\to S^1$.

By the weak factorisation theorem, $\Mlogloc$ is independent of the resolution up to the equivalence relation given by Theorem \ref{thm:details-blow-up-log}. In particular, we deduce from \ref{cor:homotopyequiv}, that the homotopy type of $\sign\tilde f:\Mlogloc\to S^1$ does not depend on the choice of $\mu$. \\

For the complete Milnor fibration, we may proceed in the same way.
Set $$\Mcplloc\coloneq\Mcpl\cap\pr^{-1}\left(\mu^{-1}(x_0)\right)$$
We have a log-geometric characterisation $\Mclogloc \simeq \Mcplloc$ where
$$\Mclogloc=\left\{(x,\varphi,\psi)\,:\,(x,\varphi)\in\Mtoploc,\,\psi\in\Hom_{\mon}\left(\M_x,(\Ro,\cdot)\right),\,\forall g\in\OO_x^*,\,\psi(g)=|g(x)|\right\}.$$
By functoriality, $\tilde f$ induces a continuous map $\sign\tilde f:\Mclogloc\to S^1$ which extends $\sign\tilde f:\Mlogloc\to S^1$. This map is well-defined up to the equivalence relations given in Theorems \ref{thm:details-blow-up-cpl1} and \ref{thm:details-blow-up-cpl2}.
\end{rmk}

\end{document}